\documentclass{amsart}
\usepackage{lineno}
\modulolinenumbers[5]
\usepackage{setspace}
\usepackage{hyperref}
\usepackage{amsmath}
\usepackage{amssymb}
\usepackage{amsfonts}
\usepackage{bm}
\usepackage{amsthm}
\usepackage{array,epsfig,fancyhdr}
\usepackage[sectionbib,numbers]{natbib}
\usepackage[utf8]{inputenc}
\usepackage{CJKutf8}
\usepackage{tikz}

\usepackage{xr}
\newtheorem{theorem}{Theorem}[section]
\newtheorem{lemma}[theorem]{Lemma}
\newtheorem{corollary}[theorem]{Corollary}

\newtheorem{conjecture}[theorem]{Conjecture}

\theoremstyle{definition}
\newtheorem{definition}[theorem]{Definition}
\newtheorem{example}[theorem]{Example}

\theoremstyle{remark}
\newtheorem{remark}[theorem]{Remark}

\numberwithin{equation}{section}

\DeclareGraphicsExtensions{.pdf,.png}
\usepackage{wrapfig}
\usepackage{lipsum}

\def \Gdeg {\mathrm{Gdeg}}

\begin{document}

\title[To prove the Four Color Theorem unplugged]{A renewal approach to prove\\ the Four Color Theorem unplugged\\[1.5ex]{\footnotesize Part III:}\\[0.5ex] Diamond routes, canal lines and $\Sigma$-adjustments}

\author{Shu-Chung Liu}
\address{Institute of Learning Sciences and Technologies, National Tsing Hua University, Hsinchu, Taiwan}
\email{sc.liu@mx.nthu.edu.tw}



\date{\today}



\begin{abstract}
This is the last part of three episodes to demonstrate a renewal approach for proving the Four Color Theorem without checking by a computer. The first and the second episodes have subtitles: ``RGB-tilings on maximal planar graphs'' and ``R/G/B Kempe chains in an extremum non-4-colorable MPG,'' where R/G/B stand for red, green and blue colors to paint on edges and an MPG stands for a maximal planar graph. We focus on an extremum non-4-colorable MPG $EP$ in the whole paper. In this part we introduce three tools based on RGB-tilings. They are diamond routes, normal and generalized canal lines or rings and $\Sigma$-adjustments. Using these tools, we show a major result of this paper:  no four vertices of degree 5 form a diamond in any extremum $EP$.  
\end{abstract}     

\setcounter{section}{13}
\setcounter{figure}{32}

\maketitle

\section{Diamond routes} \label{sec:DR} 

This section and the next one are independent. We introduce a method to build a green (or red/blue) tiling on an MPG step by step. At the beginning of this introduction, we shall first assume the existence of a green tiling. Given an MPG or a semi-MPG, say $M$, with a green tiling $T_g:E(M)\rightarrow \{\text{green, black}\}$, we associate every green edge, say $e_i$, with a green $e_i$-diamond most of the time; but a green $e_i$-triangle if $e_i$ is along an outer facet of $M$. 

   \begin{definition}
Given $M$ with a green green tiling $T_g$, a \emph{green diamond route} $\mathbf{dr}_g$ in $M$ is a sequence $\mathbf{dr}_g:=(e_1, e_2, \ldots, e_k)$ of \underline{distinct} green edges such that every consecutive pair $e_i$- and $e_{i+1}$-diamonds or triangles share a common black edge. 
   \end{definition}

   \begin{definition}
Continue from previous definition. We may consider $\vec{\mathbf{dr}}_g:=(e_1\rightarrow e_2\rightarrow \ldots\rightarrow e_h)$ as a \emph{directed} green diamond route in $M$. In the directed mode, the triangles involved in this particular route are separated into two classes:  \emph{out-triangles} and \emph{in-triangles}.
   \end{definition}

For instance, in Figure~\ref{fig:directedDR} we have a 7-semi-MPG with a green tiling which unfortunately has a green 5-cycle, namely $C_5:=v_5$-$v_6$-$v_7$-$v_c$-$v_a$-$v_5$. The two graphs show two different ways to demonstrate $(e_1, e_2, \ldots, e_6)$ by marking edges as sequence and $\vec{\mathbf{dr}}_g$ by a  green-gray dashed line with direction.  
   \begin{figure}[h]
   \begin{center}
   \includegraphics[scale=0.9]{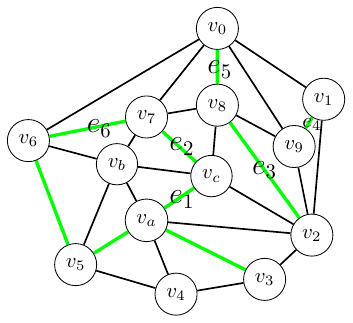}
   \includegraphics[scale=0.9]{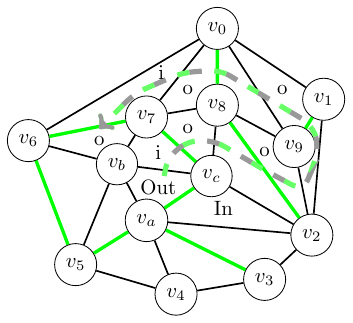}
   \end{center}
   \caption{$\mathbf{dr}_g:=(e_1, e_2, \ldots, e_6)$ and  $\vec{\mathbf{dr}}_g:=(e_1\rightarrow e_2\rightarrow \ldots\rightarrow e_6)$}  \label{fig:directedDR}
   \end{figure}
On the second graph of Figure~\ref{fig:directedDR}, we also indicate the two triangles of $e_1$-diamond as the \emph{initial-in} (marked by ``In'') and the \emph{initial-out} (marked by ``Out''). And then there are some in-triangles (marked by ``i'') and out-triangles (marked by ``o'') along the green-gray dashed line $\vec{\mathbf{dr}}_g$. Actually we miss several ``i'' to mark, because if any green triangle is marked by ``o'' then the triangle on the other side of the its  $e_i$-diamond must be the corresponding in-triangle.

It is nature to define a \emph{green diamond ring} in $M$.  For instance, let us use the first graph in Figure~\ref{fig:directedDR} again. We find that $\mathbf{dr}_g:=(e_2, e_3, e_4, e_5, e_6, e_2)$ forms a green diamond ring. Rings are good to switch the roles of in-triangles and out-triangles \underline{for themselves}.  

\subsection{Green diamond route vs green canal lines}

In Section~\ref{RGB1-sec:Grandline} we introduced canal lines with each them along two parallel canal banks. At that time we mentioned that we shall treat all triangle as nodes and came out the traditional idea about the dual graph of $M$. Now let us formally define our \emph{dual graph} of $(M;T_g)$ or $(M;T_{rgb})$, denote by $DG(M;T_g)$ or $DG(M;T_{rgb})$, which is a little bit different from the traditional one. 

Let us use Figure~\ref{fig:DGM} to explain. The set of nodes $V(DG(M;T_g)$ consists of all triangle facets (circles) and some \emph{pseudo nodes} (rectangles) near-by and along every outer facet. Notice that the traditional dual graph will set only one node for every outer facet but we rather set $k$ pseudo nodes for a $k$-gon outer facet. Every link of $DG(M;T_g)$ crosses exactly an edge $e$ in $E(M)$, and we shall color every link according to the color of $e$ in $T_g$. By the similar way, $DG(M;T_{rgb})$ can be defined. 
\begin{figure}[h]
\begin{center}
   \includegraphics[scale=0.9]{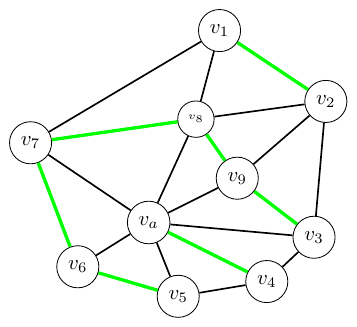}
   \includegraphics[scale=0.9]{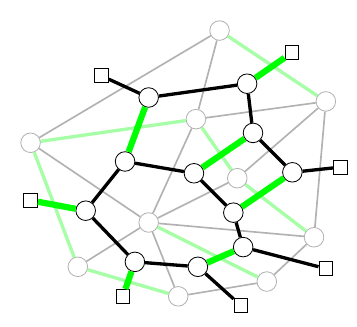}
\end{center}
\caption{vertices and edges in $(M;T_g)$; nodes and links in $DG(M;T_g)$} \label{fig:DGM}
\end{figure}   

Given $DG(M;T_g)$, a green canal line $gCL_i$ is a path going through black links in $DG(M;T_g)$. If $DG(M;T_{rgb})$ is provided, a green canal line $gCL_i$ is a path going through red and blue links alternately. All these $gCL_i$ are destined when $T_g$ or $T_{rgb}$ is given.

Given $DG(M;T_g)$, a green diamond route $\mathbf{dr}_g$ or directed one $\vec{\mathbf{dr}}_g$ is a path going through green and black links alternately. If $DG(M;T_{rgb})$ is provided, all red and blue edges are treated as black.  These $\mathbf{dr}_g$ or $\vec{\mathbf{dr}}_g$ are various for picking one of two black edges many times.

Although $DG(M;T_g)$ is much clear to show diamond route and canal lines, turning to this new graph it really is bothering us; so we keep using $DG(M;T_g)$ and $DG(M;T_{rgb})$.

\subsection{Amending $T_g$ by a green diamond route or ring}
In our study, there are two major ways to amending $T_g$ or $T_{rgb}$, and if possible we wish to break up unwelcome green odd-cycles. The first way is what we are going to introduce, and the second way is discussed in Section~\ref{sec:Gcanallines}. Let us see the following twp examples.

   \begin{example}
Let us adopt the 7-semi MPG $M$ and $T_g(M)$ in Figure~\ref{fig:directedDR} as the original setting, and we are going to break the green odd-cycle $C_5:=v_5$-$v_6$-$v_7$-$v_c$-$v_a$-$v_5$ in this original $T_g(M)$. Applying edge-color-switch along $\mathbf{dr}_g:=(e_2, e_3 \ldots, e_6, e_2)$, we get the first graph in Figure~\ref{fig:Amending1}. In addition, denote $e_0:=v_3v_a$, and $e_{-1}:=v_5v_a$. Applying edge-color-switch along $\mathbf{dr}'_g:=(e_1, e_0, e_{-1}, e_1)$, we get the second graph in Figure~\ref{fig:Amending1}.
   \begin{figure}[h] 
   \begin{center}
   \includegraphics[scale=0.9]{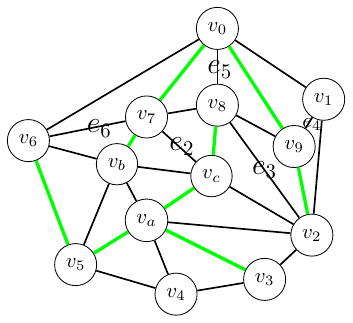}
   \includegraphics[scale=0.9]{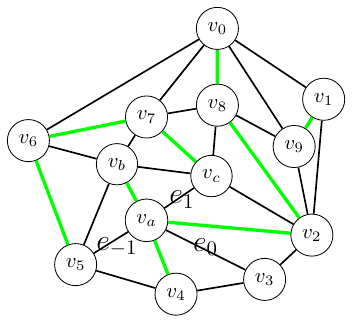}
   \end{center}
\caption{Eliminate a green odd-cycle} \label{fig:Amending1}      
   \end{figure}
Let $e_7:=v_5v_6$. Applying edge-color-switch along $\mathbf{dr}''_g:=(e_0, e_1, e_2, \ldots, e_7)$, we get the first graph in Figure~\ref{fig:Amending2}. Applying edge-color-switch along $\mathbf{dr}'''_g:=(e_4, e_5, e_6, e_7)$, we get the second graph in Figure~\ref{fig:Amending2}.
   \begin{figure}[h]
   \begin{center}
   \includegraphics[scale=0.9]{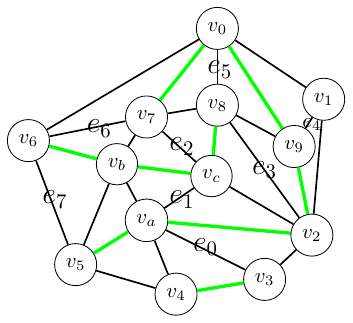}
   \includegraphics[scale=0.9]{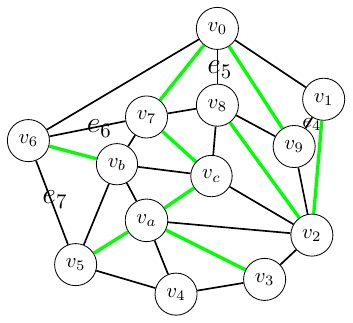}
   \end{center}
\caption{Eliminate a green odd-cycle, more examples} \label{fig:Amending2}         
  \end{figure}   
All these four processes tear the green 5-cycle $C_5$ and create new odd-cycle-free tilings in $M$. Since $M$ is One Piece, each of these four new green tilings is grand and
offers its own 4-coloring function.    
   \end{example}

   \begin{example}
We have a $(7,5)$-semi-MPG $M$ in Figure~\ref{fig:Amending3} and an original G-tiling $T_g(M)$ as the first graph. We also see a green odd-cycle $C_7$ in $T_g(M)$. This green odd-cycle is along the annular shape of $M$. After the first edge-color-switch process along the the green-gray dashed line $\mathbf{dr}_g$, we obtain the middle graph with a new green tiling $T'_g$. Unfortunately $T'_g$ is not grand.   
   \begin{figure}[h]
   \begin{center}
   \includegraphics[scale=0.68]{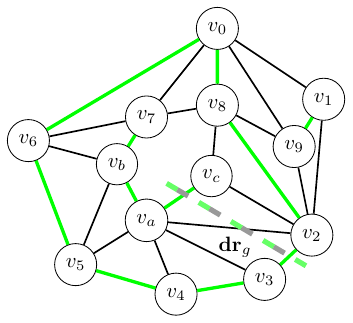}
   \includegraphics[scale=0.68]{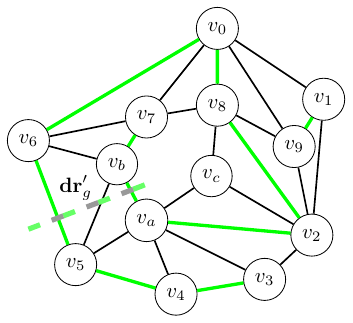}
   \includegraphics[scale=0.68]{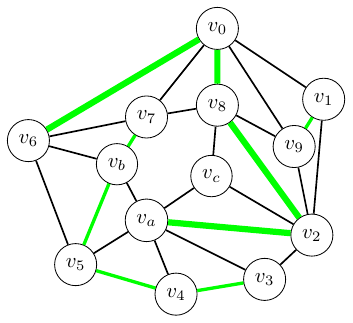}
   \end{center}
\caption{Eliminate a green odd-cycle; the middle is not grand} \label{fig:Amending3}         
  \end{figure}
Keep going! Let us apply the second edge-color-switch process along the the green-gray dashed line $\mathbf{dr}'_g$, and then obtain $T'_g$ as the third graph. Finally we achieve a grand G-tiling$^\ast$ (the abbreviation of ``G-tiling without any green odd-cycle'').
   \end{example}

   \begin{remark}
By the first example, we experience a constructing method to transform a normal green tiling with odd-cycles into a G-tiling$^\ast$. However, this constructing method does not guarantee an odd-cycle-free result, because some other new green odd-cycles might be created after this edge-color-switch process. By the second example, we experience that the grand property might be destroyed. To avoid this awkward situation, we shall choose  a green diamond route that crosses the green odd-cycle in and out. Fortunately, most of time we deal with One Piece; thus any green diamond route must be in-and-out w.r.t.\ every green odd-cycle.
   \end{remark}

\section{Orientation by an initiator} \label{sec:Orientation}

In the previous section, particularly in the second graph of Figure~\ref{fig:directedDR}, we introduced the \emph{initial-in-triangle} (marked by ``In'') and the \emph{initial-out-triangle}  (marked by ``Out''); and then we indicate \emph{in-triangles} (marked by ``i'') and \emph{out-triangles} (marked by ``o'') along a given directed green diamond route  $\vec{\mathbf{dr}}_g:=(e_I=e_0\rightarrow e_1\rightarrow \ldots\rightarrow e_k)$. Usually, we need only mark all out-triangles by ``o'' and ignore ``i''. 

The initial-in-triangle and initial-out-triangle form the \emph{initial-$eI$-diamond}. However, only an initial-$eI$-diamond can have two possible directions by switching $\triangle\text{In}$ and $\triangle\text{Out}$. Thus, as for Figure~\ref{fig:directedDR} we had $eI=e_1$ and we shall denote  $\triangle\text{Out}:=(eI, v_b)_\text{O}$ as the initial-out-triangle (sometimes we just use \text{Out} without a $\triangle$ in from of it), and $\triangle\text{In}:=(eI, v_2)_\text{I}$ as the initial-in-triangle. Once we assign an initial-in-triangle or an initial-out-triangle, we can generate many different green diamond routes and each directed green diamond route $\vec{\mathbf{dr}}_g$ offers its class of out-triangles and in-triangles. 
   \begin{example}   \label{exp:orien}
Let us adopt the right graph in Figure~\ref{fig:directedDR} first, but remove the original directed green diamond route on that graph. In Figure {fig:inoutTriangles}, we build another two routes starting at $(eI, v_b)_\text{O}$. By these two directed green diamond routes together with the route given in the right graph in Figure~\ref{fig:directedDR}, we obtain 10 out-triangles in total that are shown by the right graph in Figure~\ref{fig:inoutTriangles}.   
   \begin{figure}[h]
   \begin{center}
   \includegraphics[scale=0.9]{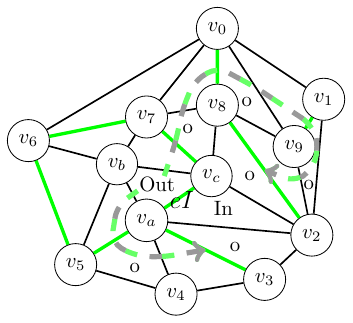}
   \includegraphics[scale=0.9]{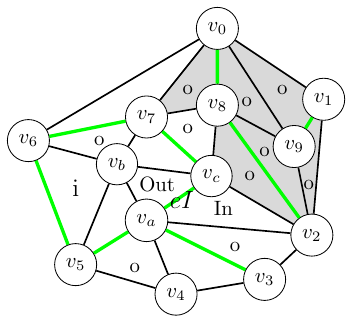}        
   \end{center}
   \caption{$OT(e_I, v_b)_\text{O}=\{\text{out-triangles generated by $(eI, v_b)_\text{O}$}\}$}    
   \label{fig:inoutTriangles}     
   \end{figure}
Most of in-triangles lay on the opposite side of out-triangles with their corresponding green edges in middle. Some exceptional in-triangles happen, for instance, $\triangle v_5v_6v_b$ is an in-triangle without its corresponding out-triangle.  Obviously, these exceptional in-triangles only happen when a green edge lay on the outer facet of $M$. One more interesting observation: there are two out-triangles $\triangle v_av_2v_3$ and $\triangle v_cv_2v_8$ ``\emph{adjacent to}'' the initial-in-triangle $(e_I, v_2)_\text{I}$. This means in $T_g$ at least two green diamond rings passing through $e_I$. Since $e_I$ lays on the green 5-cycle $C_5=v_a$-$v_c$-$v_7$-$v_6$-$v_5$-$v_a$, we now have at least two edge-color-switch processes to break $C_5$.
   \end{example}

With the help of Example~\ref{exp:orien}, we experience the \emph{orientation} of triangles generated by a fixed initial-out-triangle $\triangle\text{Out}$. Given an MPG or semi-MPG $M$ with a green tiling $T_g(M)$, let us choose an particular green edge $eI$ and one of its associating triangle $(eI, u)$ to be the \emph{initial-out-triangle} $\text{Out}:=(e, u)_\text{O}$), and in the same time  \emph{initial-in-triangle} $\triangle\text{In}:=(e, v)_\text{I}$ is chosen unless $e$ is along a outer facet of $M$. A $e$-diamond or $e$-triangle is \emph{reachable} if there is a directed green diamond routes $\vec{\mathbf{dr}}_g=(e_I=e_1\rightarrow e_2\rightarrow \ldots\rightarrow e_k=e)$ for some $k\ge 1$  and all $e_i$ distinct.  If $e$ is not along a outer facet of $M$, at the final two steps of $\vec{\mathbf{dr}}_g$ as we reaching $e$, we see an in-triangle (i) first and then an out-triangle (o). Particularly when $k=1$ and the length of $\vec{\mathbf{dr}}$ is zero that means $(eI, v)$ and  $(eI, u)$ are the only reachable in- and out-triangles respectively by this $\vec{\mathbf{dr}}$. Let denote the sets 
  \begin{eqnarray*}
IT(eI, u)_\text{O}&:=&\{\text{in-triangles generated by $(eI, u)_\text{O}$, including $\triangle\text{In}$}\};\\
OT(eI, u)_\text{O}&:=&\{\text{out-triangles generated by $(eI, u)_\text{O}$, including $\triangle\text{Out}$}\}.  
  \end{eqnarray*}
Let $Tri(M)$ consists of all triangles of $M$. 
We denote the following three disjoint subsets of $Tri(M)$:
   \begin{eqnarray*}
BiT(eI, u)_\text{O}= \{\text{bi-oriented}\} &:=& OT(eI, u)_\text{O} \cap  IT(eI, u)_\text{O};\\
NonT(eI, u)_\text{O}= \{\text{non-oriented}\} &:=& Tri(M)-(OT(eI, u)_\text{O} \cup  IT(eI, u)_\text{O});\\
UniT(eI, u)_\text{O}= \{\text{uni-oriented}\} &:=& Tri(M)-BiT(eI, u)_\text{O}-NonT(eI, u)_\text{O}.
   \end{eqnarray*}
In other worlds, a \emph{non-oriented} triangle or diamond is never reachable by any $\vec{\mathbf{dr}}_g$; a \emph{bi-oriented} triangle or diamond can be reachable from two different directions.

   \begin{example}   \label{exp:orient2}
Let us still use the right graph in Figure~\ref{fig:inoutTriangles} with $\triangle\text{Out}=(eI,v_b)_\text{O}$ in Example~\ref{exp:orien}. We have  
$BiT(eI, u)_\text{O}=\{\text{triangles inside $v_0$-$v_1$-$v_2$-$v_c$-$v_8$-$v_7$-$v_0$}\}$ and $NonT(eI, u)_\text{O}=\emptyset$. We shade the region of this $BiT(eI, u)_\text{O}$ by gray. Now we show another two examples in Figure~\ref{fig:BiTNonTUniT}. On the first graph, we assign $\triangle\text{Out}:=(eI',v_4)_\text{O}$.
   \begin{figure}[h]
   \begin{center}
   \includegraphics[scale=0.9]{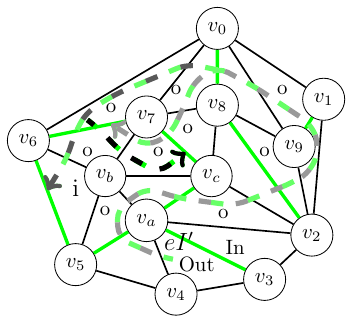}
   \includegraphics[scale=0.9]{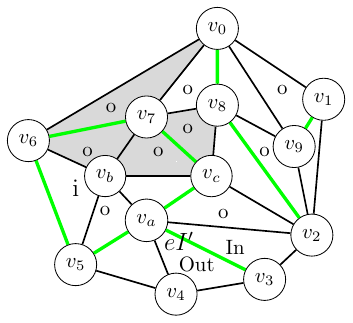}
   \includegraphics[scale=0.9]{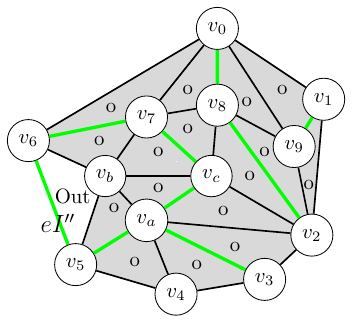}    \caption{$BiT$ in gray}    
   \label{fig:BiTNonTUniT}    
   \end{center}   
   \end{figure}
We demonstrate three directed green diamond routs to determine  that $BiT(eI', v_4)_\text{O}=\{\text{triangles inside $v_0$-$v_7$-$v_8$-$v_c$-$v_b$-$v_6$-$v_0$}\}$ and $NonT(eI', v_4)_\text{O}=\emptyset$. We can see the gray area for $BiT(eI', v_4)_\text{O}$ in the second graph. The third graph is interesting, where we assign $\triangle\text{Out}=(eI'', v_b)_\text{O}$ and we find that $OT((eI'', v_b)_\text{O})=Tri(M)$. That means starting at $(eI'', v_b)_\text{O}$ and there is a $\vec{\mathbf{dr}}_g$ reaching any triangle as an out-triangle. In this way, $BiT(eI', v_4)_\text{O}=Tri(M)-\{(eI'', v_b)\}$, i.e., almost all triangles are bi-oriented. There is a hidden meaning: Suppose $M$ is a subgraph of $\hat{M}$ and suppose a green diamond route from outside of $M$ enters through the gate $eI''$. Since all triangles in $Tri(M)$ are bi-oriented 
except $(eI'', v_b)$. Say $\triangle v_iv_jv_k$ is bi-oriented and $v_iv_j$ is along the outer facet of $M$. We can build an extended $\vec{\mathbf{dr}}_g$ that goes out through the gate $v_iv_j$ and then back into $\hat{M}-M$. Back to the first graph. If the gate $v_2v_3$ is the entrance, then a possible exit can be any of $v_3v_4$, $v_5v_6$, $v_6v_0$ and $v_0v_1$, where $v_1v_2$ and $v_4v_5$ are impossible.     
   \end{example}

\subsection{Constructing RGB-tilings or G-tilings by diamond routes} 
\label{sec:mudland}

By our method, to prove the Four Color Theorem we start with the false assumption $EP\in e\mathcal{MPGN}4 \neq \emptyset$, and then use the follow-up property $EP-\{e\}$ is 4-colorable rather than use $EP-\{v\}$ is 4-colorable with particularly $\deg(v)=5)$. A 4-colorable $EP-\{e\}$ means the existence of RGB-tilings on $EP-\{e\}$, and the existence of G-tilings on $EP$ each of which is almost a G-tiling$^\ast$ but a green odd-cycle passing through $e$.      

The truth of the Four Color Theorem for all planar graphs guarantees that all MPG's have their own RGB-tilings. However, without using the Four Color Theorem, it is independent problem that all MPG's have their own RGB-tilings.

Besides the existence property of RGB-tilings or G-tilings, constructing methods are also interesting. We try to build a green tiling on an MPG or an an $(n_1,n_2,\ldots,n_k)$-semi-MPG from initially no green edges at all. The author have achieved several results on this constructing method to build a green tiling on $M$ by using $OT(e, u)_\text{O}$ . The results will be collected in another paper in the near future.

\section{Generalized canal lines and Kempe chains}  \label{sec:Gcanallines}

We are now introducing the second topic of this article--\emph{generalized canal lines}. The study is pretty long and we just use a pentagon in $EP$ to demonstrate the idea in this section.

The reader shall see the previous Sections~\ref{RGB2-sec:tanglingProperty} and~\ref{RGB2-sec:ediamond}  to review normal \emph{canal lines}. Basically, given an MPG or $(n_1,n_2,\ldots,n_k)$-semi-MPG $M$, a \emph{grand} R-tiling and red grand canal system are same idea of two representations (See Lemma~\ref{RGB1-thm:grandRtilingGrandLineSys}). First of all, we need an R-tiling $T_r:E(M)\rightarrow \{\text{red, black}\}$. The following is a brief review of a \emph{grand} R-tiling and a \emph{grand} R-canal system. 
  \begin{itemize}
\item A grand R-tiling: It is \emph{grand} one if the vertex set $V(M)$ can be partitioned into two disjoint parts $V_1$ and $V_2$ such that 
the subgraph $G_{bl}$ of $M$ induced by all black edges is a bipartite graph on bipartite-sets $V_1$ and $V_2$, and also there is no red edge between $V_1$ and $V_2$. (Most of time, we will draw red edges in $V_1$ thicker than $V_2$.)   
\item A grand R-canal system: It is \emph{grand} one if we can arrange orientation for all R-canal lines such that the flow directions are parallel but opposite on the two sides of each red edge. 
  \end{itemize}   

When an RGB-tiling is provided, a \emph{normal} red canal line is also simultaneously a G-diamond route and a B-diamond route that \underline{follows the orientation} of R-canal system, i.e., follows the two sides of red river banks. R/G/B are actually symmetric and switchable under some circumstances. In the previous sections we used G-diamond routes and here we introduce R-canal lines, because a green light in traffic means Go and free to cross; a red light in traffic means STOP and no crossing. However, a green diamond route in a provided RGB-tiling is not necessary a red canal line. Between any two consecutive green edges along a green diamond route, it could be a red edge or a blue one.   

A \emph{generalized canal line or ring} mainly follows the orientation of R-/G-canal system but crosses some particular red/green edges occasionally. In the following, we will use several example to explain how to operate a generalized canal line or ring.  
  
Please, refer to some figures in Section~\ref{RGB1-sec:RGB4coloring} for examples and counterexamples. Briefly we use ``R-tiling$^\ast$'' to stand for the abbreviation of ``R-tiling without any red odd-cycle.'' Notice that a grand R-tiling is not necessary an R-tiling$^\ast$. The study of the rest paper highly depends on 
Sections~\ref{RGB1-sec:RGB4coloring}, \ref{RGB1-sec:Grandline}, \ref{RGB2-sec:tanglingProperty} and \ref{RGB2-sec:NecSufConds}.

   \begin{lemma}   
Let $M$ be an MPG or an $n$-/$(n_1,n_2,\ldots,n_k)$-semi-MPG with an R-tiling $T_r$.  
      \begin{enumerate}
\item[(a)] A red tiling $T_r:=\bigcup_i rC_i$ on $M$ and a red canal system $rCLS:=\bigcup_j rCL_j$ are different perspectives of looking the same thing.
\item[(b)] By linking nodes of triangles, a red canal line $rCL_j$ of $T_r$ is either {\rm (b1)} a close cycle, called \emph{canal ring}, or {\rm (b2)} a path starting from one outer facet and ending at another outer facet (maybe the same outer facet), while the pair of entrance and exit on the two end of this path are both black edges along the outer facets. 
\item[(c)] If $M$ is an MPG, then every red canal line $rCL_j$ is ring. If $M$ is an $n$-semi-MPG, then the connection of entrances and exits
of this red canal system $rCLS$ creates a non-crossing match among all black edges along the unique outer facet.   
      \end{enumerate}      
   \end{lemma}

\subsection{The rotation of the dual Kempe chains w.r.t.\ $(EP, v)$ with $\deg(v)=5$} \label{sec:rotationDualKempe}

To prove Four Color Theorem, approaching by contrapositive method is nearly inevitable. We shall assume $e\mathcal{MPGN}4$ nonempty and deal with a pseudo extremum graph, say $EP$, which is minimum in cardinality among all non-4-colorable MPG's.  The classical Kempe's proof consider a 5-semi-MPG defined by $P:= EP-\{v\}$, where $v$ is any vertex in $EP$ with $\deg(v)=5$. Our approach simulate Kempe's classical proof: Given the same situation as the setting of Kempe's proof, we consider a 4-semi-MPG $Q_i:= EP-\{v v_i\}$ for the five neighbors  $v_1, v_2,\ldots, v_5$ of $v$. By Theorem~\ref{RGB1-thm:eMPG4}, any non-trivial subgraph of $EP$ is 4-colorable; so $P$ and $Q_i$ are all 4-colorable. 

The classical Kempe's proof applies vertex-color-switching, and our renewal method uses edge-color-switching. Please, refer to Sections~\ref{RGB2-sec:RGBKempeChainDeg5}, \ref{RGB2-sec:tanglingProperty}, \ref{RGB2-sec:ediamond} and~\ref{RGB2-sec:NecSufConds} for details. 

Let us introduce the first main idea involved generalized canal rings: \emph{the rotation of the dual Kempe chains} w.r.t.\ $(EP, v)$. The idea is demonstrated by Figure~\ref{fig:KempeCRotation} briefly. Given any extremum planar graph $EP\in e\mathcal{MPGN}4$, there are at least 12 vertices of degree 5 (see Theorem~\ref{RGB1-thm:V5}, and also in \cite{Kempe1879}). Let $v\in V(EP)$ with $\deg(v)=5$ and $v_1,\ldots,v_5$ be its five neighbors. Let us denote $\Omega:=v_1$-$v_2$-$\ldots$-$v_5$-$v_1$, and $\Sigma$ ($\Sigma'$) to be the sub-area or sub-graph of $EP$ inside (outside) of $\Omega$. 

We will prove that all six graphs in Figure~\ref{fig:KempeCRotation} with their own RGB-tiling  on $EP-\{vv_i\}$ for $i=1,\ldots, 5$ in Figure~\ref{fig:KempeCRotation} are different  statuses of a same congruence class. Please, refer to Subsection~\ref{RGB2-sec:SynEquivCong} for the definition of the three different equivalence relations: synonym ($\overset{\text{\tiny syn}}{=}$, $\langle\cdot\rangle$), equivalence ($\equiv$, $[\cdot]$ under $\Omega$) and congruence ($\cong$ under $\Omega$). 
   \begin{figure}[h]
   \begin{center}
\includegraphics[scale=0.9]{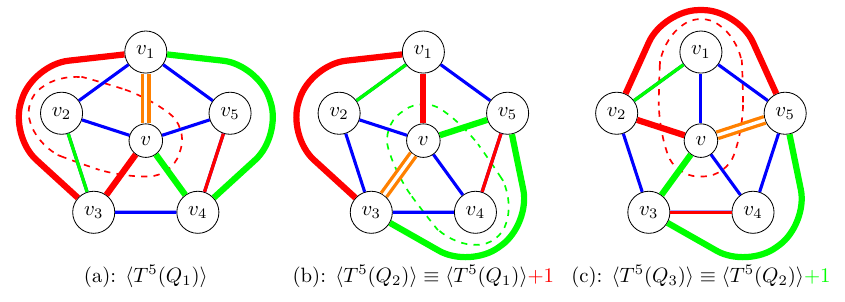}
\includegraphics[scale=0.9]{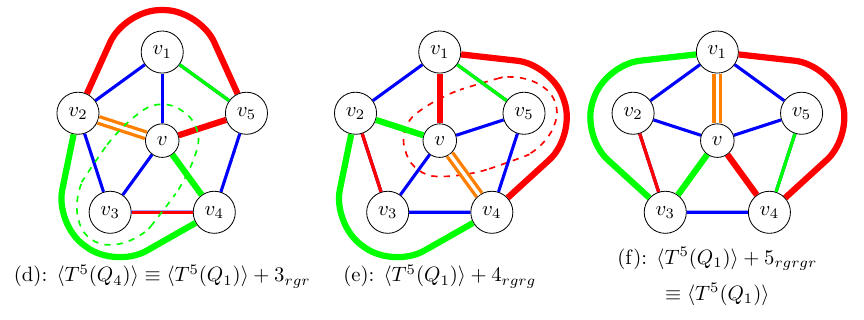}
   \end{center}
   \caption{The rotation of the dual Kempe chains w.r.t.\ $(EP; v)$} \label{fig:KempeCRotation} 
   \end{figure}

Here we focus on the this \emph{pentagon sub-area} of $EP$ (every $EP$ without exception), it is nature to name this picture of $(EP; v)$ with $\deg(v)=5$ by $T^5$ particularly. Let $\Omega:=v_1$-$v_2$-$\ldots$-$v_5$-$v_1$ and $\Sigma$ ($\Sigma'$) is the subgraph of $EP$ inside (outside) $\Omega$. Due to the shape of pentagon sub-area, we also denote $Ptg:=\Sigma$ or simply use ``$5$'' as a superscript for short, where $\Sigma$ is a general notation for all kinds of sub-area. We start with an RGB-tiling (as same as a 4-coloring function) on $Q_1:=EP-\{vv_1\}$; this RGB-tiling, denoted by $T^5(Q_1)$, is guaranteed by Theorem~\ref{RGB1-thm:eMPG4}: any non-trivial subgraph of $EP$ is 4-colorable. Moreover, we use $\langle T^5(Q_1)\rangle$ to denote the class of synonyms of $T^5(Q_1)$ which consist of six RGB-tilings on $Q_1:=EP-\{vv_1\}$ by switching (permuting) edge-colors red, green and blue all over whole $Q_1$. We will use the equivalence class $[T^5(Q_1)]$ for this kind of process later as a comparison.

\begin{figure}[h]      
   \begin{center}
   \includegraphics[scale=1]{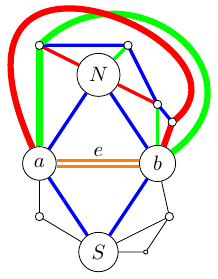}\qquad
   \includegraphics[scale=1]{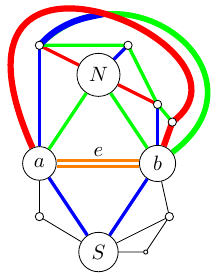}
   \end{center}
   \caption{Type A and Type B RGB-tilings for $EP-\{e\}$}  \label{fig:AtypeBtype00InSecondPaper} 
   \end{figure}
Two more things need to mention: (1) Most of time we use Type A $e$-diamond for the rest of this study; (2) $T^5(Q_1)$ is just one of many different RGB-tiling on $Q_1$ and and $\langle T^5(Q_1)\rangle$ is just one of many different  classes of synonyms. For the reader's convenience, we re-draw Type A and Type B $e$-diamonds in Figure~\ref{fig:AtypeBtype00InSecondPaper}.  

Before we start our process, let us look at (a) to (f) in Figure~\ref{fig:KempeCRotation} individually. So far these six Type A RGB-tilings $T(Q_i)$ are independent and their existence dues to Theorem~\ref{RGB2-thm:EPediamond} due to a fixed $e$-diamond. 
   \begin{remark} \label{re:EquVSSynonym}
Because the pentagon $\Sigma$ is very simple, the equivalence class $[T(Q_i)]$ is unique and it has the dual Kempe chains $(K_r, K_g)$ w.r.t.\ $(EP; vv_i)$ as the skeleton in $\Sigma'$; however $\langle T(Q_i)\rangle$ might have many different classes of six synonyms for a fix $i$. According to the pictures only, we see that (a) and (f) are equivalent, even though we just see red and green edge-colors switched. But we can not guarantee that the two underline RGB-tilings of (a) and (f) are synonyms, because the same \emph{skeleton} in $\Sigma'$ shared by both (a) and (f) might has different pair of paths. Two synonyms must share same paths of skeleton in $\Sigma'$. In this pentagon sub-area or $\Omega$, the skeleton in $\Sigma'$ can only be the dual Kempe chains $(K_r, K_g)$.
   \end{remark}

Starting with (a):$\langle T^5(Q_1)\rangle$ in Figure~\ref{fig:KempeCRotation} where we draw a \emph{red generalized canal ring}, denoted by $rGCL_1$, shown as a red dashed line. (Both red canal ring and canal line are always denoted by $rCL_i$ and never by $rCR_i$.) It is generalized because it crosses the red edge $vv_3$ and also it crosses the yellow double-line\footnote{This double-line is actually orange color because yellow color in not easy to see for publications.} $vv_1$. Now let us perform edge-color-switching (ECS for short\footnote{We use acronyms VCS and ECS to stand for ``vertex-color-switching'' and ``edge-color-switching'' respectively.}) on  $rGCL(v_1v_2)$ and then we obtain (b) in Figure~\ref{fig:KempeCRotation}. Because the new RGB-tiling $\langle T^5(Q_2)\rangle$ is obtain by perform ECS on \textcolor{red}{red canal ring}, we also write it as   
$\langle T^5(Q_1)\rangle\textcolor{red}{+1}$. 

So, how to perform edge-color-switching on (or along) a red generalized canal ring? It is very simple:
   \begin{itemize}
\item Switch edge colors of green and blue, just like what we do for a normal red canal line;  
\item Switch edge colors of red and yellow double-line, and this switching rule is what we perform for the ``generalized'' segment.
\item To perform edge-color-switching on (or along) a green/blue generalized canal ring, we just apply the two items above, under symmetry of three colors. 
   \end{itemize} 

In total, we perform ECS five times in Figure~\ref{fig:KempeCRotation} and alternately using red/green generalized canal rings. We use different ways to remark these six graphs. For instance, (e) is actually $\langle T^5(Q_5)\rangle$; however when we follow the previous four processes,  $\langle T^5(Q_1)\rangle + 4_{rgrg}$ is a good way to represent this equivalence class.  

   \begin{remark}
After explore these five processes and the remarks labeled for (a) to (f), we find that the class $[T^5(Q_i)]$ is much better than the class $\langle T^5(Q_i)\rangle$. We will use former one in the rest of the paper.
   \end{remark}

   \begin{remark}
If we start with (f) and perform ECS five more times, then we can get a new $[T^5(Q_1)]$, even though this new $[T^5(Q_1)]$ might have different underline RGB-tiling compared with the original (a): $T^5(Q_1)$, i.e., we have two RGB-tilings of a same equivalence class but not necessary the same.    
   \end{remark}
   
By definition of congruence relation defined in Subsection~\ref{RGB2-sec:SynEquivCong}, we have 
  $$
[T^5(Q_1)]_\Omega \cong [T^5(Q_2)]_\Omega \cong \cdots \cong [T^5(Q_5)]_\Omega.
  $$  
where the subscript $\Omega$ means the sub-area that $[\cdot]$ builds up. Therefore, we have the following theorem.  

  \begin{theorem}
Let $EP\in e\mathcal{MPGN}4$, $v\in V(EP)$ with $\deg(v)=5$ where the five neighbors of $v$ form $\Omega:=v_1$-$v_2$-$\ldots$-$v_5$-$v_1$. Under the equivalence relation $[\cdot]_\Omega$, all RGB-tilings on $EP$ are in a same congruence class.
  \end{theorem}
\noindent
Be careful! Once we focus on another topic for discussion as well as different $\Omega'$, then another equivalence relation $[\cdot]_{\Omega'}$ presents; so this theorem does not necessarily hold.  

   \begin{remark}
There is another reason that we had better use $[\cdot]$, rather than $\langle \cdot\rangle$. Again, let us focus on (a) which is both in Figures~\ref{fig:KempeCRotation} and~\ref{fig:KempeCRotation2}. This time we apply the other red generalized canal ring drawn as the red dashed line in Figures~\ref{fig:KempeCRotation2}. Yes, there are only two \emph{major}\footnote{Only \emph{major} red generalized canal rings (by our method) or \emph{major} red-connected components (by Kempe's method) can effect the skeleton in $\Sigma'$. Please, see Section~\ref{RGB2-sec:tanglingProperty} for details.} red generalized canal rings to reach a Type A $vv_2$-diamond shown as (b) and (b') in the two figures. Visually it is clear that (b) and (b') just a same graph with green/blue switched. However, (b) and (b') are not necessary same class of synonyms, but we are sure that (b) and (b') have same skeleton in $\Sigma'$; therefore the graphs of (b) and (b') are equivalent. One more thing shall be kept in mind: $K_r|_{v_1}^{v_3}$ represents a red-connected component connecting $v_1$ and $v_3$ that means $K_r|_{v_1}^{v_3}$ might contain a bunch of red paths from $v_1$ to $v_3$. Any red canal ring inside $K_r|_{v_1}^{v_3}$ is not \emph{major}. Also notice that    
  $$
[T^5(Q_1)\textcolor{red}{+1}]:=
[\langle T^5(Q_1)\rangle\textcolor{red}{+1}]
  $$ 
and the left-hand-side is our standard notation.
   \begin{figure}[h]
   \begin{center}
\includegraphics[scale=0.9]{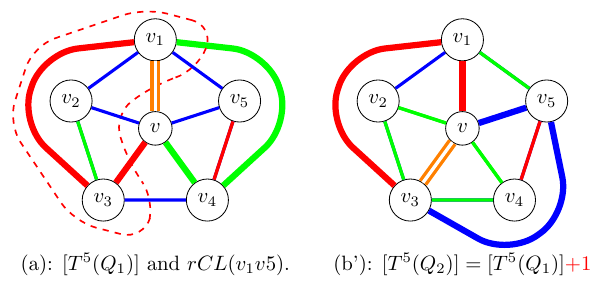}
   \end{center}
   \caption{The other major red generalized canal ring for (a)} \label{fig:KempeCRotation2} 
   \end{figure}
   \end{remark}

   \begin{definition} \label{def:GCLconjugate}
The two major red generalized canal rings shown in the the first graph in Figures~\ref{fig:KempeCRotation} and in~\ref{fig:KempeCRotation2} are usually denote by $rGCL(v_1v_2)$ and $rGCL(v_1v5)$, because they come out of $\Sigma$ from edges $v_1v_2$ and $v_1v5$ respectively. We say $rGCL(v_1v_2)$ and $rGCL(v_1v5)$ are \emph{conjugate} for the results, (b) and (b'), of ECS on each of them are equivalent. If there are three or more \emph{major} red generalized canal rings, the idea of \emph{conjugation} is more complicate. We will talk about it then.  
   \end{definition}

\subsection{More concepts about $Ptg$}
   
There are still two concepts to explore $Ptg$.  The first one is a 4-colorable function on $Ptg$ locally or this $Ptg$ is a sub-area of a 4-colorable MPG. The only possible representative is shown as the first graph in Figure~\ref{fig:PtgColorable}.  One of its features is the setting of the four red edges.     
   \begin{figure}[h]
   \begin{center}
\includegraphics[scale=0.76]{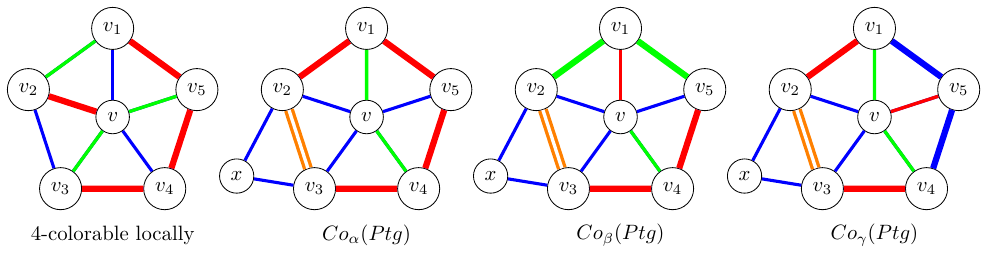}
   \end{center}
   \caption{4-colorable locally and a yellow double-line on $\Omega$} \label{fig:PtgColorable} 
   \end{figure}

The second concept comes from a half $e$-diamond of Type A in $Ptg$ shown as the rest three graphs in Figure~\ref{fig:PtgColorable}.

   \begin{lemma} \label{thm:Coalpha}
If $EP$ has $Co_\alpha(Ptg)$, then $\deg(v_2, v_3)\ge 6$.
   \end{lemma}
   \begin{proof}
If $\deg(v_2)=5$ (or $deg(v_3)=5$) then this vertex has to follow the tangling property w.r.t.\ a degree 5 vertex in $EP$. However, $K_r|_{v_2}^{v_3}$ is too simple to interest with $K_g$. Thus, $\deg(v_2)\neq 5$.  Please, see Lemma~\ref{RGB2-thm:deg5tangling}.
   \end{proof}
   \begin{figure}[h]
   \begin{center}
\includegraphics[scale=0.82]{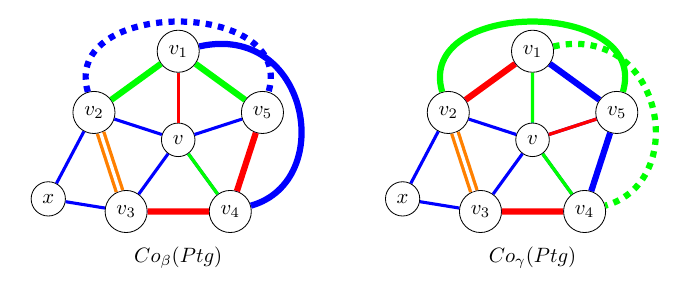}
   \end{center}
   \caption{Good $Co_\beta(Ptg)$/$Co_\gamma(Ptg)$, and 4-colorable ones} \label{fig:PtgColorable2} 
   \end{figure}

   \begin{corollary} \label{thm:Coalpha2}
If $EP$ has $Co_\beta(Ptg)$ with $\deg(v_2)=5$ or $\deg(v_3)= 5$, then $K_b|_{v_2}^{v_5}$ is impossible and $K_b|_{v_1}^{v_4}$ must exist. Please, refer to Figure~\ref{fig:PtgColorable2}.
    \end{corollary}

   \begin{corollary} \label{thm:Coalpha3}
If $EP$ has $Co_\gamma(Ptg)$ with $\deg(v_2)=5$ or $\deg(v_3)= 5$, then $K_g|_{v_1}^{v_4}$ is impossible and $K_g|_{v_2}^{v_5}$ must exist. Please, refer to Figure~\ref{fig:PtgColorable2}.
   \end{corollary}

\section{Kempe chains around two adjacent vertices of degree 5 in $EP$} \label{sec:rotationDualKempe2}
Given a particular $EP\in e\mathcal{MPGN}4$ who has two adjacent vertices, say $a$ and $b$, of degree 5, us will perform ECS on generalized canal ring around $a$ and $b$ to obtain many Kempe chains in different statuses.

Let $T\!D:=(\{a,b\};\deg(a,b)=5)$. It is the \emph{topic for discussion}. Around $T\!D$ is the \emph{boarder} $\Phi$ as a cycle.  Here we have $\Phi:=v_1$-$v_2$-$c$-$v_4$-$v_5$-$d$-$v_1$.\footnote{Here we use $\Phi$ to distinguish from $\Omega$ in Subsection~\ref{sec:rotationDualKempe}.} The formation definition of $T\!D$ will be given in Subsection~\ref{sec:TD}. Because $\deg(a,b)=5$, we also use ``$55$''  to stand this particular $EP$. The two graphs in Figure~\ref{fig:KempeCRo2deg5basic} are the initial RGB-tilings of $EP$ with Type A $ab$-diamond under equivalence. The subscripts $\alpha$ and $\beta$ are just labels to distinguish them. Clearly, $[T_\alpha^{55}]\not\equiv [T_\beta^{55}]$. We will show that $[T_\alpha^{55}]\cong [T_\beta^{55}]$. Another obvious observation is that the equivalence class $[T_\alpha^{55}]$ of RGB-tilings is symmetric w.r.t.\ the vertical line and the horizontal one; so is the class $[T_\beta^{55}]$ with more imagination.
   \begin{figure}[h]
   \begin{center}
   \includegraphics[scale=0.9]{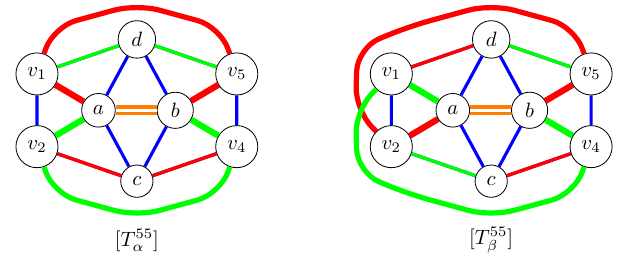}
   \end{center}
   \caption{The two initial RGB-tilings of $EP$ with $T\!D =55$} \label{fig:KempeCRo2deg5basic} 
   \end{figure}

For $[T_\alpha^{55}]$ there are two major red (green) generalized canal rings, namely $rGCL(dv_1)$ and $rGCL(v_1v_2)$ ($gGCL(cv_2)$ and $gGCL(v_1v_2)$). For $[T_\beta^{55}]$ there are two major red (green) generalized canal rings, namely $rGCL(v_1v_2)$ and $rGCL(cv_2)$ ($gGCL(v_1v_2)$ and $gGCL(dv_1)$). Of course, all of them are conjugate in pairs.

\subsection{Let us rock-n-roll} \label{sec:RnR}
Starting with the initial status $S_0:=[T_\alpha^{55}]$, let us perform 10 consecutive processes of ECS according to those red/green dashed lines drawn in Figure~\ref{fig:KempeCRo2deg5}.  The whole figure shows the rotation of many Kempe chains around vertex $a$ and $b$, or around $\Phi$.
   \begin{figure}[h]
   \begin{center}
   \includegraphics[scale=0.82]{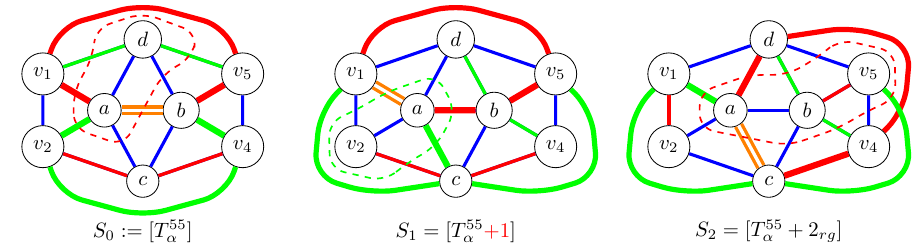}
   \includegraphics[scale=0.82]{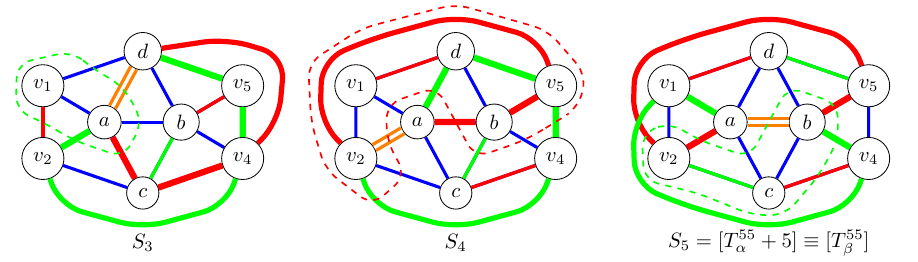}
   \includegraphics[scale=0.82]{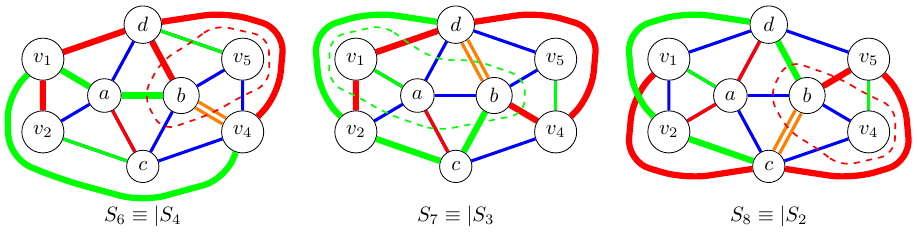}
   \includegraphics[scale=0.82]{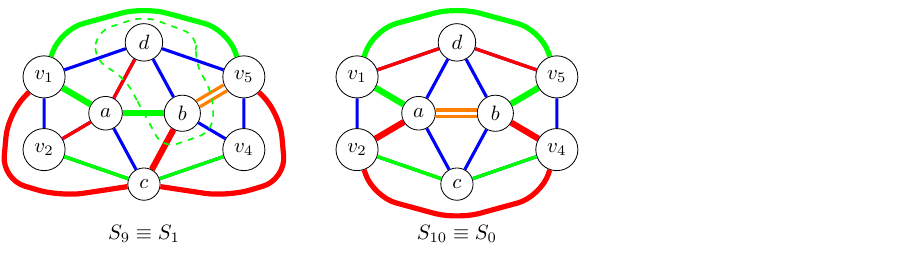}
   \end{center}
   \caption{Rock-n-roll around $(\{a,b\};\deg(a,b)=5)$} \label{fig:KempeCRo2deg5} 
   \end{figure}

   \begin{remark}
Some details in Figure~\ref{fig:KempeCRo2deg5} need to mention. We have two versions of $rGCL(dv_1)$, where we do show the one that turns around vertex $a$, and the other one that turns around vertex $b$. That is why we have both $K_g|_c^{v_1}$ and $K_g|_c^{v_5}$ in $S_1$. Also notice that even though the graph we draw for $S_1$ seems to have $\Gdeg(c)=3$, but it is more possible that $\Gdeg(c)=2$ for $K_g|_c^{v_1}$ and $K_g|_c^{v_5}$ sharing one green edge in $\Sigma'$. So far we still have $\deg(c)\ge 5$.
   \end{remark}

   \begin{remark}
The Kempe chain $K_r|_d^{v_4}$ in $S_2$ and $S_3$ can be replaced by $K_r|_d^{c}$, because we can only guarantee that red edge $cv_4$ is red connected with $d$. If it is really $K_r|_d^{c}$, then $\deg(c)\ge 6$.  The same thing happens in $S_3$ and $S_4$ for $K_g|_{v_2}^{v_4}$; it can be either replace by $K_g|_{v_2}^{v_5}$ or $K_g|_{v_2}^{d}$. If it is really $K_g|_{v_2}^{v_5}$, then $\deg(v_5)\ge 7$. If it is really $K_g|_{v_2}^{d}$, then $\deg(d)\ge 6$. There are more discussion by this similar idea, and we will talk about it then.
   \end{remark}

   \begin{remark}
The notation $|S_4$ in the third line of this figure means reflection of $S_4$ w.r.t.\ the vertical line. There is also notation $\underline{S_\ast}$ that means reflection of $S_\ast$ w.r.t.\ the horizontal line. 
   \end{remark}   

   \begin{remark}
For $S_1$ there are three major green generalized canal rings, namely $gGCL(v_1v_2)$ (the given green dashed line in this figure) $gGCL(v_4v_5)$ and $gGCL(dv_1)$. So, what is the idea of \emph{conjugation} now? The reader can check that ECS on $gGCL(v_1v_2)$ is conjugate with $gGCL(v_4v_5)\oplus gGCL(dv_1)$, where $\oplus$ means combination or connecting these two generalized canal line in a proper way; also $gGCL(v_4v_5)$ is conjugate with $ gGCL(v_1v_2)\oplus gGCL(dv_1)$. We could choose ECS on $gGCL(v_4v_5)$ as our second process to change $S_1$ and then obtain $S'_2$. Clearly, $S'_2\equiv |S_2$; so the rest after $S'_2$ are all reflections w.r.t.\ the horizontal line.  Even though $gGCL(dv_1)$ is conjugate with $gGCL(v_1v_2)\oplus gGCL(v_4v_5)$, it seems useless; because applying ECS on $gGCL(dv_1)$ does not make a single $e$-diamond of Type A but two of Type B. Please, see Figure~\ref{fig:S1andS8} and refer to the next remark.
   \end{remark} 

According to Figure~\ref{fig:KempeCRo2deg5}, we have 
  \begin{equation} \label{eq:T55S09}
[S_0=T_\alpha^{55}]_\Phi \cong [S_1]_\Phi \cong \cdots \cong [S_5=T_\beta^{55}]_\Phi \cong \cdots \cong [S_{10}]_\Phi\equiv [T_\alpha^{55}]_\Phi.
  \end{equation}

   \begin{remark}  \label{re:Sx0}
Since $S_8\equiv |S_2$ and $S_9\equiv S_1$, something happens in between $S_1$ and $S_2$, as well as in between $S_8$ and $S_9$, are similar.  In the last remark, we said that performing ECS on  $gGCL(dv_1)$ in $S_1$ seems useless. Let us try it, as well as performing ECS on  $rGCL(v_1v_2)$ in $S_8$. We obtain $S_x \cong S_y$ and they have two Kempe chains of only one color. Notice that the pair of yellow double-lines in $S_x$ are different from the ones in $S_y$. However, the congruence of them is decided by the edge-coloring along $\Phi$, denoted by $Co(\Phi)$, and the skeleton in $\Sigma'$. The difference inside $\Sigma$ between  $S_x$ and $S_y$ associates with $\Sigma$-adjustment that will be introduced later.      
Now we can explore this interesting $S_x$. 
   \begin{figure}[h]
   \begin{center}
   \includegraphics[scale=0.9]{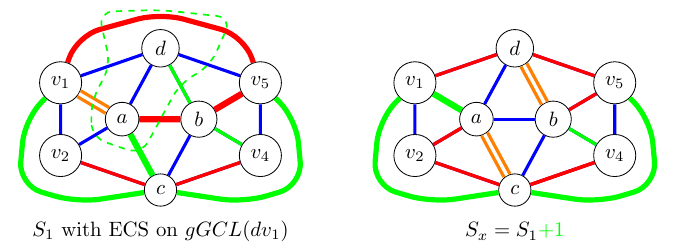}
   \includegraphics[scale=0.9]{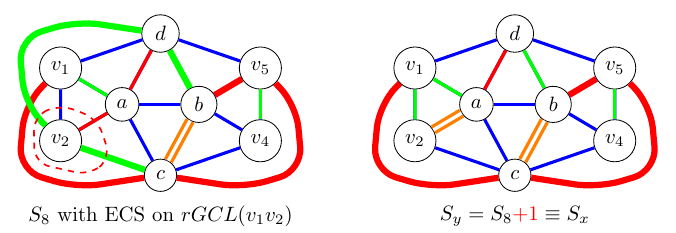}
   \end{center}
   \caption{About $S_1$ and $S_8$} \label{fig:S1andS8} 
   \end{figure}
   \end{remark}  

According to Figure~\ref{fig:S1andS8} , we have 
  \begin{equation} \label{eq:T55Sx}
[S_0=T_\alpha^{55}]_\Phi \cong [S_x]_\Phi \equiv  [S_y]_\Phi.
  \end{equation}

We use notation $|S_\ast$ and $\underline{S_\ast}$ to denote the reflection images of $S_\ast$ w.r.t.\ the vertical line and the horizontal one respectively. Also $|\underline{S_\ast}$
is reflected twice. Let $[S_\ast]_\text{sym}$ consists of the equivalence of the these four reflection images of $S_\ast$. Most of time $[S_\ast]_\text{sym}$ has four different elements, but both $[T_\alpha^{55}]_\text{sym}$ and $[T_\beta^{55}]_\text{sym}$ have only one element. Due to this fact of only one element and Equations~\ref{eq:T55S09} and~\ref{eq:T55Sx}, we derive the next lemma.

   \begin{lemma}   \label{thm:CongruentEachOther}
All elements in $\{[S_0], [S_1], \ldots, [S_9], [S_x]\}_\text{sym}$, where the subscript $\text{sym}$ means this set consists all symmetric images w.r.t.\ the vertical line and the horizontal one, are  congruent to each other.
   \end{lemma}

Here is one more amazing and important property. 

   \begin{theorem} \label{thm:allGraphsSimilarToEP}
Let $EP\in e\mathcal{MPGN}4$ with $T\!D:=(\{a,b\};\deg(a,b)=5)$ and it is drawn as the underlining graph shown in Figure~\ref{fig:KempeCRo2deg5basic}. Also we adopt the notation $\Phi$, $\Sigma$ and $\Sigma'$. Let us fix the subgraph $\Sigma'$ and consider all kinds of MPG's, denoted by $M_\ast$, such that $M_\ast$ has exactly two vertices inside $\Phi$. Among all these $M_\ast$, only $EP$ is non-4-colorable.
   \end{theorem}
   \begin{proof}
We sill name the two vertices inside $\Phi$ by $a$ and $b$. The first MPG that we consider is $M_1$ show as the first (underlining) graph in Figure~\ref{fig:M1andM2}. By the existence of RGB-tiling $S_x$ in $\Sigma'$ and the setting of  
   \begin{figure}[h]
   \begin{center}
   \includegraphics[scale=0.83]{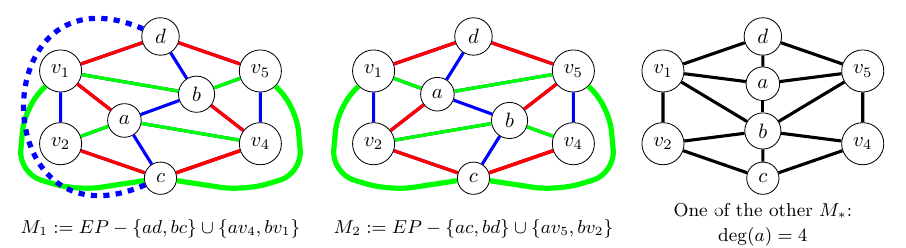}
   \end{center}
   \caption{$M_1$, $M_2$ and the rest of 4-colorable $M_\ast$} \label{fig:M1andM2} 
   \end{figure}
edge-coloring inside $\Phi$, we prove that $M_1$ is 4-colorable. Notice that the green cycle in $M_1$ is even, and even if the possible blue dashed $K_b|_c^d$ exists we see a blue even-cycle. The second graph in Figure~\ref{fig:M1andM2} MPG $M_2$ is the reflection of $M_1$ w.r.t.\ the vertical line. Clearly $M_2$ is 4-colorable. 

Among all MPG's $M_\ast$, including some graphs might have edges linking vertices in $\{c,d,v_1,v_2,v_4,v_5\}$ (for example, given edge $v_1v_4$), only $EP$, $M_1$ and $M_2$ can keep $\deg(a,b)= 5$; the rest MPG's of $M_\ast$ must have $\deg(a)\le 4$ or $\deg(b)\le 4$. Since $|M_\ast|=\omega$, which is the same order of all extremum planar graphs in $e\mathcal{MPGN}4$, the rest MPG's of $M_\ast$ must be 4-colorable by Corollary~\ref{RGB1-thm:V5more2}. The proof is complete.
   \end{proof}

   \begin{lemma}  \label{thm:2deg5}
Let $EP\in e\mathcal{MPGN}4$ with $T\!D:=(\{a,b\};\deg(a,b)=5)$ and it is drawn as the underlining graph shown in Figure~\ref{fig:KempeCRo2deg5basic}. The ten graphs, from $S_0$ to $S_9$ in Figure~\ref{fig:KempeCRo2deg5}, as well as $S_x$ and $S_y$ in Figure~\ref{fig:S1andS8}, are congruent. So we can only deal with one of them in discussion of 4-colorable or not; because each of them is a necessary condition for $EP$ being non-4-colorable. 
   \end{lemma}

Let us still fix the subgraph $\Sigma'$ and consider three vertices $\alpha$, $\beta$ and $\gamma$ inside $\Phi$. We focus on two particular MPS's: $M^+_a$ and $M^+_b$ in Figure~\ref{fig:MaandMb}  where the superscript ``+'' means $|M^+_a| = |M^+_b| = \omega +1$. We dare to ask a question: Are $M^+_a$ and $M^+_b$ 4-colorable? The answer is yes, if non-4-colorable $(EP;55)$ do exist. We will prove this interesting problem later. 
   \begin{figure}[h]
   \begin{center}
   \includegraphics[scale=0.9]{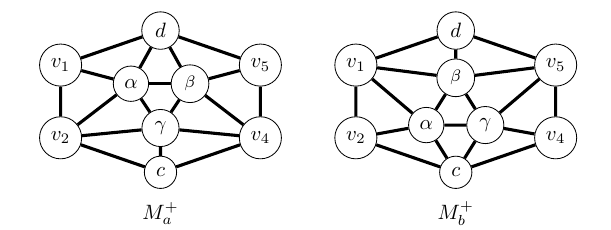}
   \end{center}
   \caption{$M^+_a$ and $M^+_b$} \label{fig:MaandMb} 
   \end{figure}

\subsection{4-colorable MPG's with $(\{a,b\};\deg(a,b)=5)$} \label{sec:4colorableDeg5Adj}

Let us think reversely. We focus on an 4-colorable MPG, say $M$, of any order with $T\!D:=(\{a,b\};\deg(a,b)=5)$, i.e., the underlining graph of $M$ is as same $\Sigma$ as the ones in Figure~\ref{fig:KempeCRo2deg5basic}. Of course, $M$ is a different $\Sigma'$ compared with $EP$. For $M$, we name $\Omega:=v_1$-$v_2$-$c$-$v_4$-$v_5$-$d$-$v_1$ as the boarder between $\Sigma$ and $\Sigma'$. 

   \begin{figure}[h]
   \begin{center}
   \includegraphics[scale=0.85]{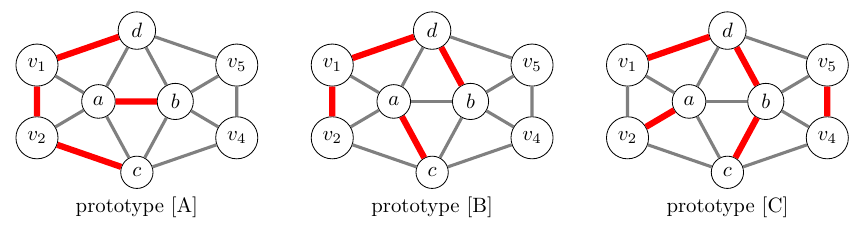}
   \end{center}
   \caption{Prototypes of red edge-coloring around vertex $a$} \label{fig:M4colorable} 
   \end{figure}
   \begin{figure}[h]
   \begin{center}
   \includegraphics[scale=0.85]{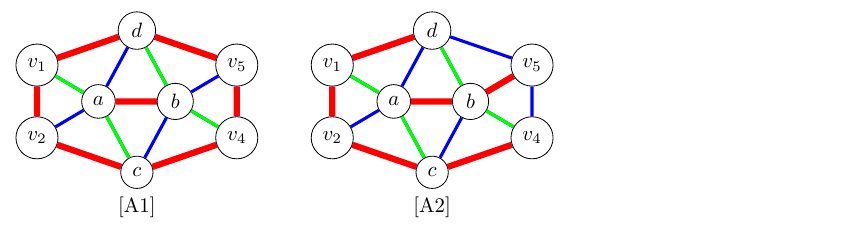}
   \includegraphics[scale=0.85]{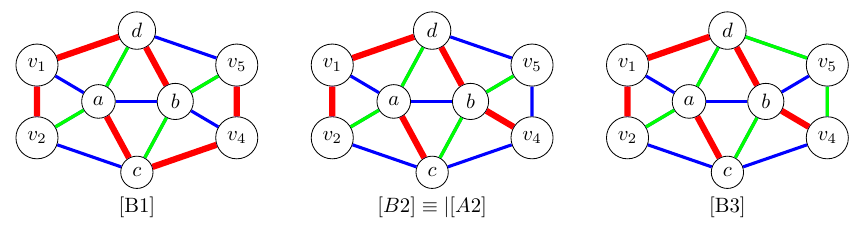}
   \includegraphics[scale=0.85]{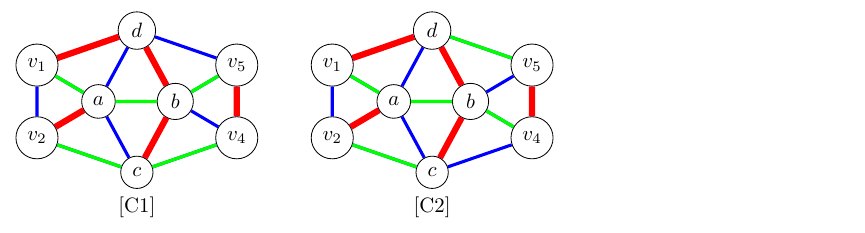}
   \end{center}
   \caption{All possible types of RGB-tilings on $M$} \label{fig:M4colorable2}  
   \end{figure}   
   
Because $\Sigma'$ is symmetric w.r.t.\ the vertical line and the horizontal one, we will only explore those representatives of RGB-tilings on $M$. We can first consider all possible R-tilings on $\Sigma$ to force $M$ 4-colorable.  We get three prototypes of red edge-coloring around vertex $a$ in Figure~\ref{fig:M4colorable}. The details are given in Figure~\ref{fig:M4colorable2}.

Depending on what $M$ is, at least one of these six patterns of RGB-tiling (including their symmetric patterns) on $\Sigma$ can extend to $\Sigma'$, and then we can fulfill the assumption that $M$ is 4-colorable. Any kind of possible Kempe chains in $\Sigma'$, which are prepared for these six patterns of RGB-tiling on $\Sigma$, cannot create any R/G/B odd-cycle.     

   \begin{remark}
We can drop $[B2]$ from the list in Figure~\ref{fig:M4colorable2}, because $[B2] \equiv |[A2]$.      
   \end{remark}
   
   \begin{lemma} \label{thm:NoIntersection}
(a): There is no intersection between $\{[S_0], [S_1], \ldots, [S_9], [S_x]\}_\text{sym}$ and $\{[A1], [A2], [B1], [B3], [C1], [C2]\}_\text{sym}$, where the subscript means these two sets consist all symmetric images w.r.t.\ the vertical line and the horizontal one. (b): Particularly there is not intersection between their own $Co(\Phi)$. Otherwise, $EP$ with $T\!D:=(\{a,b\};\deg(a,b)=5)$ is 4-colorable.
   \end{lemma}   
   \begin{proof}
Let us name $ATLAS_N: =\{[S_0], [S_1], \ldots, [S_9], [S_x]\}_\text{sym}$ temporarily, but later we will modify this set without hurting this lemma.
Also $ATLAS_4: = \{[A1], [A2], [B1], [B3], [C1], [C2]\}_\text{sym}$. Clearly, the subscripts $N$ and $4$ stand for non-4-colorable and 4-colorable. 

For (a), they should no intersection. It need patience to check these two sets in the coming subsection.     
Also, no intersection is a necessary condition for non-4-colorablility, but not a sufficient condition.
 
As for (b), it raises another important question: Does the edge-color along $\Phi$, namely $Co(\Phi)$, unique determine the 4-colorable property of $M:= \Phi \cup \Sigma \cup \Sigma'$? However, this new question is far more than what we claim only for $(EP;55)$. 

Sorry! This brief remark is not a real proof. The real proof of (a) is in the next subsection. Part (b) is just a consequence of (a), because the comprehensive study on $ATLAS_\ast$ do show that the types of $Co(\Phi)$ uniquely determine each element in $ATLAS_N\cap ATLAS_4$  and more that that. The property of unique determination so far only works for $(EP;55)$ and $(EP;Ptg)$.      
   \end{proof}

\subsection{$ATLAS$ of Figures~\ref{fig:KempeCRo2deg5}, \ref{fig:S1andS8}, \ref{fig:M4colorable2} and more}  \label{sec:ATLAS}

For convenience, {\bf all $[\ast]$ in this subsection is actually $[\ast]_\text{sym}$}. The main purpose of this subsection is to list all kinds of $Co(\Phi)$ under synonyms and symmetries; then we can offer a proof of  Lemma~\ref{thm:NoIntersection}(a) about $ATLAS_N$ and $ATLAS_4$. 

According to Subsections~\ref{sec:RnR} and~\ref{sec:4colorableDeg5Adj}, it is nature to ask: Does the union set of those $Co(\Phi)$ in Figures~\ref{fig:KempeCRo2deg5}, \ref{fig:S1andS8} and~\ref{fig:M4colorable2} consists of all possible $Co(\Phi)$ provided that $Co(\Sigma')$ is 4-colorable. The answer is no, but this union set consists of nearly all. 

By Lemma~\ref{RGB1-thm:evenoddRGB}, the array $(\#r,\#g,\#b)$ of numbers of red, green and blue edges along $Co(\Phi)$ can only be $(0,0,6)$, $(0,2,4)$ and $(2,2,2)$ under synonyms. The cases of $(0,0,6)$ and $(0,2,4)$ are easy and demonstrated by Table~\ref{tb:summary} systematically.
\begin{table}[h]
\centering
\begin{tabular}{ |c|c|c|c| } 
 \hline
\includegraphics[scale=0.85]{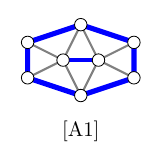} 
&
\includegraphics[scale=0.85] {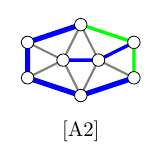}
& 
\includegraphics[scale=0.85] {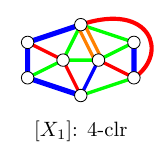} 
& 
\includegraphics[scale=0.85] {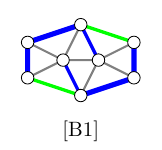}
\\
 \hline 
\includegraphics[scale=0.85] {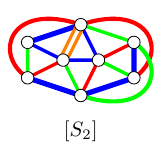}
&
\includegraphics[scale=0.85] {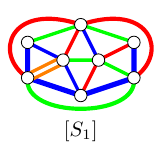}
& 
\includegraphics[scale=0.85] {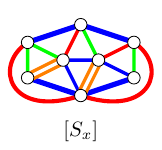} 
& { }
\\ 
 \hline
\end{tabular}\vspace{3mm}
\caption{$Co(\Sigma)$ and skeletons in $\Sigma'$  for $(0,0,6)$ and $(0,2,4)$} \label{tb:summary}
\end{table}
Notice that in this table skeletons are unnecessary for cases [A$\ast$], [B$\ast$] and [C$\ast$], because they are for colorable for any kinds of $K_\ast$ in $\Sigma'$. However, we do draw a $K_r$ for case $[X_1]$, which never showed up before\footnote{That is why we mark it by ``X''.}. We also claim $[X_1]$ 4-colorable, because of Lemma~\ref{thm:Coalpha}.  

To list all cases of  $(2,2,2)$ under synonyms, we can refer to cases of $(0,2,4)$ and then choose two blue edges along $\Phi$. But this way is so tedious and  twice the work with half the results. Let us consider the two edges in same color adjacent or not (Y/N). So we shall follow four extra requirements: $(Y, Y, Y)$, $(N, Y, Y)$, $(N, N, Y)$ and  $(N, N, N)$. Now we will demonstrate all cases in Table~\ref{tb:summary2}.      
\begin{table}[h]
\centering
\begin{tabular}{ |c|c|c|c| } 
 \hline
\includegraphics[scale=0.85]{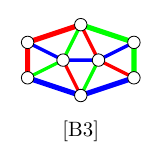} 
&
\includegraphics[scale=0.85] {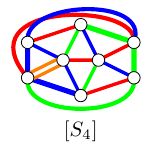}
& 
\includegraphics[scale=0.85] {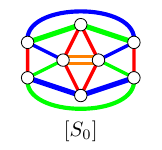} 
& 
\includegraphics[scale=0.85] {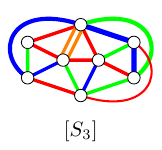}
\\
 \hline 
\includegraphics[scale=0.85] {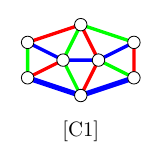}
&
\includegraphics[scale=0.85] {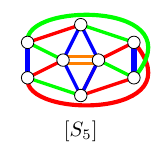}
& 
\includegraphics[scale=0.85] {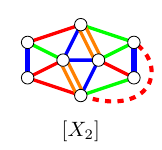} 
&
\includegraphics[scale=0.85] {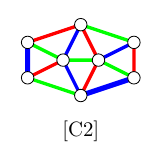}
\\ 
 \hline
\end{tabular}\vspace{3mm}
\caption{$Co(\Sigma)$ and skeletons in $\Sigma'$ for $(2,2,2)$} \label{tb:summary2}
\end{table}
Notice that there is only one case for $(Y, Y, Y)$ under synonyms.  We obtain $[X_2]$, which never showed up before. Of course this $[X_2]$ is special. Later we will find it ubiquitous in further study. The two yellow double-lines are two $e$-diamonds of Type B involving three edge-colors; therefore we can not get any information on $\Sigma'$ from these two $e$-diamonds. It is not a problem comes two or more $e$-diamonds. For instance, $[S_x]$ in Figure~\ref{fig:S1andS8} has two $e$-diamonds. The good thing is these two $e$-diamonds of Type B involving two edge-colors. Even if $S_x$ comes out of nowhere, rather than what we just showed that it is from $[S_1]$, we still can build up $K_g|_c^{v_1}$ and $K_g|_c^{v_5}$. 

   \begin{remark} \label{re:X2}
In the graph $[X_2]$ given in Table~\ref{tb:summary2}, we draw a dashed $K_r|_c^{v_5}$ on purpose. Actually, once [$X_2$] appears it must have (a): either $K_r|_c^{v_5}$ or $K_r|_d^{v_4}$; and (b): either $K_g|_c^{v_1}$ or $K_g|_d^{v_2}$. It is possible that four kinds of combinations of (a) and (b) all suit for this $EP$ with $T\!D:=(\{a,b\};\deg(a,b)=5)$, but at least one combination exists. Without loss of generality, we assume $K_r|_c^{v_5}$ appears in $[X_2]$. Then we process ECS on $rGCL(cv_4)$ and then obtain the second graph in Figure~\ref{fig:X2S3}. The second graph ($[S_3]$, Type B) and the third  graph ($[S_3]$, Type A) have same $Co(\Phi)$, but only Type A can guarantee two Kempe chains: $K_g|_d^{v_2}$ and $K_b|_d^{v_4}$. From the second graph to the third graph, the process can be done by ECS or directly by \emph{$\Sigma$-adjustment}, which is a modification inside $\Sigma$ and will be introduced later formally. Let us make  the conclusion of this remark: If $[X_2]$ exists, then $[X_2]\cong [S_3]$. But so far the existence is not guaranteed.      
   \begin{figure}[h]
   \begin{center}
   \includegraphics[scale=0.82]{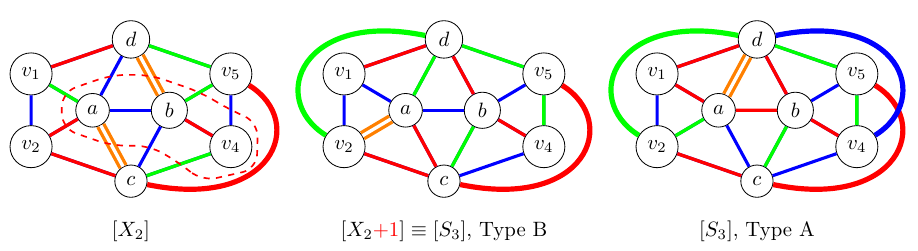}
   \end{center}
   \caption{$[X_2]$ and $[S_3]$} \label{fig:X2S3} 
   \end{figure}
   \end{remark}
   \begin{remark} \label{re:X2Sx}
Our curiosity on $[X_2]$ has not ended yet. There are two possible blue Kempe chains, namely $K_b|_{v_1}^{v_5}$ and $K_b|_{c}^{d}$, and at least one exists\footnote{Be careful! They never co-exist; they might exist for different RGB-tilings.}. Therefore, we can perform two possible ECS and then obtain $[S_5]$ and $[S_x]$ shown as the second graph and the third one in Figure~\ref{fig:X2S5Sx}. Not we realize that $[X_2]$ connects to many different colleagues in $ATLAS_N$      
   \begin{figure}[h]
   \begin{center}
   \includegraphics[scale=0.82]{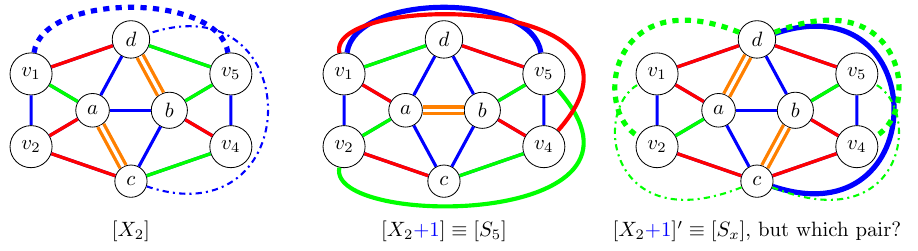}
   \end{center}
   \caption{$[X_2]$ and $[S_3]$} \label{fig:X2S5Sx} 
   \end{figure}   
   \end{remark}
   \begin{remark} \label{re:Sx}
The second graph in Figure~\ref{fig:X2S5Sx} has $K_b|_{v_1}^{v_5}$. It does not means that every $[S_5]$ equips with $K_b|_{v_1}^{v_5}$. It is possible that another $S_5$ RGB-tiling has $K_b|_{c}^{d}$. The above argument can also apply on the third graph in Figure~\ref{fig:X2S5Sx}.   
   \end{remark}
   \begin{remark} \label{re:S3}
Comparing two $[S_3]$  in Figure~\ref{fig:KempeCRo2deg5}\footnote{Please, replace $K_g|_{v_2}^{v_4}$ by $K_g|_{d}^{v_2}$; because we only guarantee this green-connectivity.} and in Figure~\ref{fig:X2S3} and under synonyms, we find that the new one (latter one) has three Kempe chains of three different colors. Remark~\ref{re:X2} offer a reason of the existence of this new $K_r|_c^{v_5}$. We provide another reason. If there is no $K_r|_c^{v_5}$, then there should be a $K_r|_d^{v_4}$ and then we can turn this $[S_3]$ to be $[X_1]$ in Table~\ref{tb:summary2}; thus this $EP\not\in e\mathcal{MPGN}4$ and a contradiction. Now we realize how interesting and important Lemmas~\ref{RGB2-thm:deg5tangling} and~\ref{thm:Coalpha} are. One more important thing: thanks for this ``another reason'', we can say $[X_2]$ is obtained from $[S_3]$ by performing ECS on $rGCL(cv_4)$ in Figure~\ref{fig:X2S3}. {\bf So, the existence of $[X_2]$ is guaranteed.} Otherwise, the process that we did in Remark~\ref{re:X2} bases on the assumption of existence of $[X_2]$ and it is possible that $[X_2]$ does not exist. Now we complete entire $ATLAS_N$ theoretically.          
   \end{remark}
   \begin{remark} \label{re:S4}
Let us re-check $[S_4]$ in Figure~\ref{fig:KempeCRo2deg5}. Using $\Sigma$-adjustment, we find a new $K_b|_{v_1}^{v_5}$. This refinement has another reason. If there is no $K_b|_{v_1}^{v_5}$, then there should be a $K_b|_d^{v_4}$ and then we can turn the original $[S_4]$ (the left graph) to be 
to be [A2];  thus this $EP\not\in e\mathcal{MPGN}4$ and a contradiction.
   \begin{figure}[h]
   \begin{center}
   \includegraphics[scale=0.85]{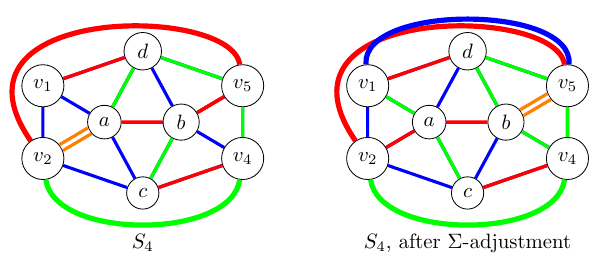}
   \end{center}
   \caption{$[S_3]$ and $\Sigma$-adjustment} \label{fig:S4} 
   \end{figure}
   \end{remark}  

Now finally we finish a proof of Lemma~\ref{thm:NoIntersection}(a). Not only a proof but also we list $ATLAS_N$ more precise with one more element $[X_2]$. Let us conclude this whole section as following theorem: 

   \begin{theorem} \label{thm:atlas}
Let $EP\in e\mathcal{MPGN}4$ with $T\!D:=(\{a,b\};\deg(a,b)=5)$. Every element  in $ATLAS_N:=\{[S_0], [S_1], \ldots, [S_5], [S_x], [X_2]\}_\text{sym}$ will appear by doing a series of ECS starting from $[S_0]$ or from any one in $ATLAS_N$.     
   \end{theorem} 
   \begin{proof}
Remark~\ref{re:S3} has already shown the existence of $[X_2]$, because it can be obtained from $[S_3]$.     
   \end{proof}

	\begin{remark}
Recall the interesting problem that we left in Figure~\ref{fig:MaandMb}. For symmetry, we only show $M^+_a$ is 4-colorable. If non-4-colorable $(EP;55)$ do exist, then the RGB-tiling $[S_x]$ exists. We only need the R-tiling of $[S_x]$ on $\Sigma'$. Put this R-tiling on the same $\Sigma'$ of $M^+_a$ and also color 5 edges inside $\Sigma$ shown in Figure~\ref{fig:MaandMb2}. Now the whole R-tiling on $M^+_a$ has no odd-cycle. 
By Theorem~\ref{RGB1-thm:4RGBtiling}, $M^+_a$ is 4-colorable.
  \begin{figure}[h]
   \begin{center}
   \includegraphics[scale=0.9]{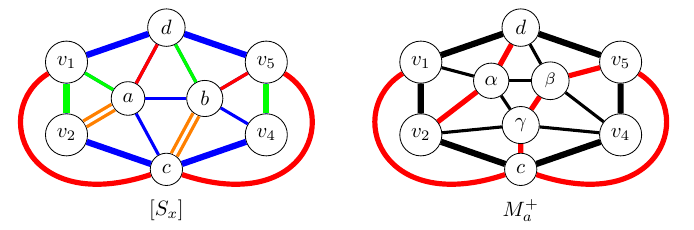}
   \end{center}
   \caption{From $[S_x]$ to $M^+_a$} \label{fig:MaandMb2} 
   \end{figure}
   \end{remark}

\section{Theory behind ECS on generalized canal rings}
\label{sec:theoryBehind}

We have already performed ECS on generalized canal rings many times, but we did not talk clearly about the theory. In this sections we explain our standard of process (SOP) to perform ECS on generalized canal lines or rings, and then investigate new Kempe chains in $\Sigma'$. 

\subsection{Topic for discussion $T\!D$ and boarder $\Omega$ in $EP$} \label{sec:TD}
Given a fixed $EP\in e\mathcal{MPGN}4$, let $T\!D:=(\{u_1, u_2,\ldots, u_k\};\text{requirements})$ consist of some chosen vertices from $EP$ as the \emph{topic for discussion}, where \emph{{requirements}} set up the particular type of $u_1, u_2,\ldots, u_k$. We also use $T\!D$ to represent the induced subgraph made by the vertices of $T\!D$. Usually we require this subgraph $T\!D$ connected and solid, where``solid'' means no hole; and it had better to be 2-connected if $k\ge 4$. The \emph{boarder} $\Omega$ (sometimes $\Phi$) is a subgraph induced by the vertex set consists of all surrounding neighbors of $T\!D$. Clearly, $\Omega$ is a cycle surrounding $T\!D$, unless $T\!D$ is so weird. Let  $\Sigma:=\Omega \cup T\!D$ and $\Sigma':=EP-T\!D$. Clearly $\Sigma\cap\Sigma'=\Omega$.  For example $T\!D:=(\{v\};\deg(v)=5)$ and $\Omega:=v_1$-$v_2$-$\ldots$-$v_5$-$v_1$ in Figure~\ref{fig:KempeCRotation}.    

Due to the definition of $e\mathcal{MPGN}4$ and Theorem~\ref{RGB1-thm:eMPG4}(b), a 4-semi-MPG $EP-\{e\}$ for any $e\in E(EP)$ is 4-colorable and there is at least an RGB-tiling on $EP-\{e\}$. By Theorem~\ref{RGB2-thm:EPediamond2}, RGB-tilings with Type A and Type B $e$-diamond both exist and are congruent to each other in pairs. Please, recall our premature concept of $ATLAS_N$ and $ATLAS_4$ in Section~\ref{sec:rotationDualKempe2}. 
To study $(EP;T\!D)$, we keep improving or refining these two sets. A \emph{primary} RGB-tiling with a particular $e$-diamond, like $T_\alpha^{55}$ and $T_\beta^{55}$ in Figure~\ref{fig:KempeCRo2deg5basic}, are the original members of $ATLAS_N$, which are two different RGB-tilings on $EP-\{ab\}$. Actually the \emph{primary}\footnote{The \emph{primary} RGB-tilings and the \emph{initial} ones are different. Please, wait for the definition of the latter ones.} RGB-tilings for the discussion on  $(EP;T\!D)$ can be all inner edges $e$ in $\Sigma$, and then we define the \emph{primary atlas} of RGB-tilings w.r.t.\ $(EP;T\!D)$: 
	\begin{eqnarray*}
ATLAS_N(EP;T\!D) & \overset{\text{\tiny tmp}}{:=}& \{T^{T\!D}\mid \text{ RGB-tilings, either Type A or Type B,}\\
& &\text{on $EP-\{e\}$ for $e\in E(\Sigma)-E(\Omega)$}\}. 	
	\end{eqnarray*}
For example, see $ATLAS_N$ in the proof of Lemma~\ref{thm:NoIntersection}.  
	
	\begin{remark}
Why did we only consider Type A for our initial $ATLAS_N$ in the last section? Because we get benefit from $T\!D$ containing some vertices of degree 5 and any Type B $e$-diamond with $e$ incident to these vertices degree 5 can be transformed to Type A $e'$-diamond without changing $Co(\Phi)$ and the associating skeleton.      
  	\end{remark}

On the other hand, to attack the non-4-colorability of $(EP;T\!D)$, we consider all possible local 4-colorable functions on $\Sigma$, each of which is also an RGB-tiling on $\Sigma$. So we construct the following set:
	\begin{eqnarray*}
ATLAS_4(T\!D) & \overset{\text{\tiny tmp}}{:=}& \{T^{T\!D}\mid \text{ RGB-tilings on $\Sigma$}\}. 	
	\end{eqnarray*}

	\begin{remark}
The main idea of whole project is to prove by contrapositive with the assumption that $EP$ exists. The existence of $(EP;T\!D)$ derives every element in $ATLAS_N(EP;T\!D)$ must exist and $ATLAS_4(T\!D)$ must be empty. Once we find  $ATLAS_N(EP;T\!D) \cap ATLAS_4(T\!D)$ non-empty, then $(EP;T\!D)$ is impossible; but not every $EP$. The edge-color $Co(\Omega)$	plays an important role on checking $ATLAS_N(EP;T\!D) \cap ATLAS_4(T\!D)=\emptyset$ or not. 
	\end{remark}

Not just for one particular \emph{abandoned} edge $e$, we might consider a particular set of \emph{abandoned} edges, denoted by $\{\ast\}$ when the elements are not chosen yet, and consider any RGB-tiling on $EP-\{\ast\}$, i.e., a combination of $e$-diamonds with Types A and B. There are two different reason to concern about a set $\{\ast\}$ containing multiple edges: 
	\begin{enumerate}
\item This new RGB-tiling on $EP-\{\ast\}$ is obtained from one of the primary RGB-tilings. For example, $[S_x]$ in Remark~\ref{re:Sx0}.
\item Sometimes we are forced to check all possible cases of $Co(\Omega)$. We might find some $Co(\Omega)$ that were not investigated yet. So, after studied they can be sorted into $ATLAS_N(EP;T\!D)$ or $ATLAS_4(T\!D)$.  For example, $[X_1]$ and $[X_2]$ in Subsection~\ref{fig:M4colorable2}. However, $[X_2]$ obtained from one of the primary RGB-tilings is still very important.  
	\end{enumerate}	   
The \emph{secondary} $ATLAS_N(EP;T\!D)$ consists of these new RGB-tilings on $EP-\{\ast\}$ (non-4-colorable on $EP$) together with all primary ones; similarly for the \emph{secondary} $ATLAS_4(T\!D)$.

Yes, we do have the \emph{tertiary} $ATLAS_N(EP;T\!D)$. In this extended collection, we consider $\{\ast\}$ to be a subset of $E(\Sigma)$; however the edges in $\{\ast\}$ still form multiple $e$-diamonds of Type A and Type B. For instance, $Co_\alpha(Ptg)$, $Co_\beta(Ptg)$ and $Co_\gamma(Ptg)$ in Figures~\ref{fig:PtgColorable} and~\ref{fig:PtgColorable2}.     

	\begin{remark}  \label{re:secondary}
Why do we always emphasize that new (secondary and tertiary) RGB-tilings on $EP-\{\ast\}$ are obtained from a primary one? Because we assume that $(EP;T\!D)$ exists, thus every primary RGB-tilings on $EP-\{e\}$ exists by Theorem~\ref{RGB2-thm:EPediamond2}. We have to guarantee all element in the secondary and the tertiary sets exist under this assumption. Therefore, $ATLAS_N(EP;T\!D) \cap ATLAS_4(T\!D)=\emptyset$ or not is really a crucial point to judge $(EP;T\!D)$. 	
	\end{remark}

\subsection{The magic of the yellow double-lines and congruence relation} \label{sec:yellowDL}
The magic of the yellow double-lines (\emph{abandoned} edges) is that replacing it by either red, green or blue, we will get an odd-cycle of the same color, which is called a \emph{Kempe chain}. Given an $e$-diamond, for Type A, there are two non-trivial Kempe chains of different colors; as for Type B, there is only one. Kempe chains are just representatives because it is possible that a clusters of red/green/blue paths linking the two end-vertices of $e$. In other words, a red Kempe chain represents red-connected property in $\Sigma'$. Please, see Section~\ref{RGB2-sec:tanglingProperty} for details.

When we have multiple $e$-diamonds, the Kempe chains in this RGB-tiling on $EP-\{\ast\}$ need to judge case-by-case. For instance,$[S_x]$ in Figure~\ref{fig:S1andS8} and $[X_2]$ in Remark~\ref{re:X2}. 

Following Remark~\ref{re:secondary}, we understand that the co-existence of certain RGB-tilings on different $EP-\{\ast\}$ built by congruence relation is so important. Congruence relation is based on performing ECS on a canal ring or a generalized canal ring. Performing ECS on a canal ring will create a new RGB-tiling on $EP-\{\ast\}$ without changing $\{\ast\}$; however performing ECS on a generalized canal ring will change to a new $\{\ast\}$. The former is easily passed on theory, but the latter need to be explained more precisely. 

Let us consider an RGB-tiling $T^{T\!D}_\alpha(EP-\{e_{\alpha 1},e_{\alpha 2},\ldots\})$,
where $\{e_{\alpha 1},e_{\alpha 2},\ldots\}$ is a set of inner edges of $\Sigma$ such that no two of them in a single triangle. Suppose a generalized canal ring $rGCL$ is a part of $T^{T\!D}_\alpha(EP-\{e_{\alpha 1},e_{\alpha 2},\ldots\})$ and it is the one we would like to perform ECS to get a new RGB-tiling $T^{T\!D}_\beta(EP-\{e_{\beta 1},e_{\beta 2},\ldots\})$. This $rGCL$ must have segments of two kinds. The first kind is along a red Kempe chain $K_r|_u^v$ (or maybe more) in $\Sigma'$.
Thus $rGCL \cap \Sigma'$ must be a normal red canal line(s) and performing ECS on $T^{T\!D}_\alpha$ will still make the new $T^{T\!D}_\beta$ still an RGB-tiling on $\Sigma'$. Notice that the tangling property in Section~\ref{RGB2-sec:tanglingProperty} might happen for this reason. Here we bring back an important question: Besides degree 5, does there any other situation have tangling property?

The segment(s) of the second kind is the part $rGCL \cap \Sigma$ and this is the real ``generalized'' part. Let us review 
the rules of ECS on a red generalized canal line or ring:
   \begin{itemize}
\item Switch edge colors of green and blue, just like what we do for a normal red canal line;  
\item Switch edge colors of red and yellow double-line, and this switching rule is what we perform for the ``generalized'' segment.
\item To perform edge-color-switching on (or along) a green/blue generalized canal ring, we just apply the two items above, under symmetry of three colors. 
   \end{itemize}      
After ECS, by the second rule above, some yellow edges in $\{e_{\alpha 1},e_{\alpha 2},\ldots\}$ turn to be red; however in view of R-tiling on $EP$, they are assumed to be red already. The remained unchanged yellow edges in $\{e_{\alpha 1},e_{\alpha 2},\ldots\}$ together with the some original red edges that passed by $rGCL$ form the new yellow-edge-set $\{e_{\beta 1},e_{\beta 2},\ldots\}$. With this new yellow-edge-set, we shall investigate new $K_g$ and $K_b$. This this the key point of our ECS process. {\bf We preserve all $K_r\cap \Sigma'$, and try to build new $K_g \cap \Sigma'$ and $K_b \cap \Sigma'$ for $T^{T\!D}_\beta$.} Now we can consider $T^{T\!D}_\beta(EP-\{e_{\beta 1},e_{\beta 2},\ldots\})$ and use a generalized canal ring $gGCL$ or $bGCL$ to perform the next ECS. The sequence of process is just like what did in Figure~\ref{fig:KempeCRo2deg5}. 

	\begin{remark}
Recall the red Kempe chain $K_r|_u^v$ described in the last two paragraph. Suppose that edges $uu'$ and $vv'$ are along $\Omega$ and crossed by $rGCL$. According the first rule above, the colors on $uu'$ and $vv'$ are just switched between green and blue, and then  $Co(\Omega)$ is changed. Thus, the sequence of ECS process can explore many different kinds of 	$Co(\Omega)$ for $(EC; T\!D)$. 
	\end{remark}

Besides Figure~\ref{fig:KempeCRo2deg5}, we also explored $ATLAS_\ast$ in Subsection~\ref{sec:ATLAS}. Without this hard work, we probably cannot find the interesting $[x_1]$ and $[X_2]$. Do we really need to explored $ATLAS_\ast$? It depends on what kind of $T\!D$ that we discuss. The more fundamental structure of $T\!D$, the more details we need to know. 

\subsection{$\Sigma$-adjustments and $\{\ast\}$}  \label{sec:SigmaAdjust}
In the last subsection we wrote: {\bf We preserve all $K_r\cap \Sigma'$, and try to build new $K_g \cap \Sigma'$ and $K_b \cap \Sigma'$ for $T^{T\!D}_\beta$.} The general way to do is observing the new set $\{e_{\beta 1},e_{\beta 2},\ldots\}$ in $\Sigma$; however, we have many different ways to build new $K_g \cap \Sigma'$ and $K_b \cap \Sigma'$ as skeleton, and we had better do our best to find them. For example, the red Kempe chain $K_r|_c^{v_5}$ of $[S_3]$, in Table~\ref{tb:summary2} or in Remark~\ref{re:S3}, is obtained by exclusive law.

The method of $\Sigma$-adjustments has  been already used in Remarks~\ref{re:Sx0}, \ref{re:X2} and~\ref{re:S4}. There are two major methods to perform a $\Sigma$-adjustment:
	\begin{enumerate}
\item Find any generalized canal ring  inside $\Sigma$ and perform ECS. Then we have new set of abandoned edges $\{e'_{\beta 1},e'_{\beta 2},\ldots\}$ in $\Sigma$ to build new $K_g \cap \Sigma'$ and $K_b \cap \Sigma'$.
\item Just keep $Co(\Omega)$ of this moment unchanged. Try to re-arrange a new 
single color tiling inside $\Sigma$. Then according $Co(\Omega)$ to complete the other two edge-coloring. Most of time, we cannot obtain an RGB-tiling on $\Sigma$ but one on $\Sigma-\{\ast\}$. Now try to build new $K_g \cap \Sigma'$ or $K_b \cap \Sigma'$.
	\end{enumerate}
The reader can practice method (1) using $[S_1]$ or $[S_4]$, and a red generalized canal ring inside $\Sigma$. All [A1], [A2], [B1], [B2], [B3], [C1], [C2] were made by method (2) without any abandoned edge. Also $[X_1]$ and $[X_2]$ were studied by method (2) at the very beginning. 

Working on $ATLAS_\ast$ in Subsection~\ref{sec:ATLAS} is so tedious, but the job in Subsection~\ref{sec:RnR} is standard.  

Because $EP$ is the extremum or the smallest non-4-colorable MPG, we can only 4-color $EP-\{e\}$ and some time $EP-\{\ast\}$. The theory of $e$-diamond is much easier, for it can be only be Type A and Type B, which are co-existing. A large set $\{\ast\}$ of abandoned edges make a coloring or an RGB-tiling on both $\Sigma$ and $\Sigma'$, as well as $Co(\Omega)$, more complicated. Also we need to be very careful: Without linking to a certain congruent RGB-tiling $T_\alpha^{T\!D}(EP-\{e\})$, we have no right to guarantee the existence of $T_\beta^{T\!D}(EP-\{\ast\})$.

\section{Three degree 5 vertices in a triangle}
\label{sec:3degree5}

Let us consider the situation that three vertices of degree 5 in $EP$ form a triangle. Our have new $T\!D:=(\{a,b,c\};\deg(a,b,c)=5)$ or ``$5^3$'' in short that consists of three vertices of degree 5 shown in Figure~\ref{fig:3deg5a} and $\Omega:=d$-$v_1$-$v_2$-$\ldots$-$v_5$-$d$.

By Lemma~\ref{thm:2deg5} we choose the first graph $[T_\alpha^{55}]$ in Figure~\ref{fig:KempeCRo2deg5basic} to discuss. The following two graphs, $[T_1^{5^3}]$ and $[T_2^{5^3}]$, are determined by $cv_3$ being blue or green. However, given $cv_3$ blue, we see a 5-cycle $K_g\cup \{ab\}$. By Lemma~\ref{thm:Coalpha} or the tangling property for $EP$ with $\deg(a,b)=5$, we have a contradiction.  
   \begin{figure}[h]
   \begin{center}
   \includegraphics[scale=0.87]{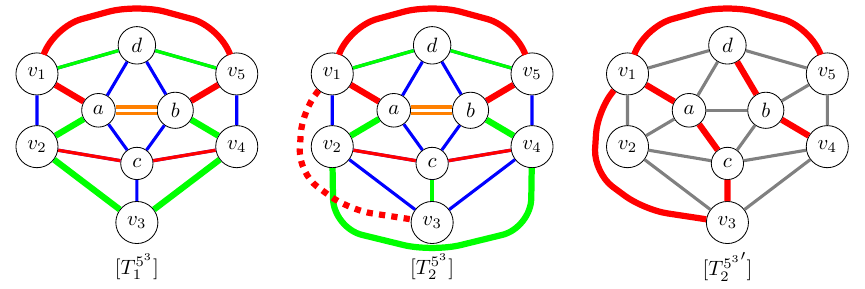}
   \end{center}
   \caption{Only $[T_2^{5^3}]$ is is good for $EP$, not $[T_1^{5^3}]$} \label{fig:3deg5a} 
   \end{figure} 
Thus, only $[T_2^{5^3}]$ can be the proper RGB-tiling for this $(EP;T\!D)$. In addition, we claim that the red dashed path exists. We simply re-arrange a new red tiling inside $\Sigma$ and treat green/blue as black. The new red tiling on $EP$ is shown as the third graph $[{T_2^{5^3}}']$ and there must have at least a red odd-cycle crossing $\Sigma$. The only way is the red path $K_r|_{v1}^{v_3}$.     

Just for fun, the reader can re-arrange another new red tiling inside $\Sigma$ by setting $d$-$a$-$v_2$ and $v_3$-$c$-$b$-$v_5$ red, and then exams the new red odd-cycle crossing $\Sigma$. We leave the reader to draw this result.

   \begin{lemma} \label{the:3deg5Representative}
Let $a, b, c$ be three vertices in a triangle of $EP$ with $\deg(a,b,c)=5$. There is only one congruent class of RGB-tilings on $EP$. This congruent class has a representative shown as $[T_2^{5^3}]$ in Figure~\ref{fig:3deg5a}. Please, turn the red dashed line in $[T_2^{5^3}]$ solid.  
   \end{lemma} 

As different way to prove the existence of the red dashed in Figure~\ref{fig:3deg5a} is given as the following process. In Figure~\ref{fig:3deg5a2}, starting with the original $[T_2^{5^3}]$,  we perform two ECS on a $rGCL$ and then on a $bCL$. The result $[T_2^{5^3}+2_{rb}]$ is given as the third graph. Since we use a red generalized canal ring crossing $\Sigma'$ and the second blue canal ring is all inside $\Sigma$. So the new red Kempe chain $K_r|_{v_1}^{v^3}$ is supposed to exist in $[T_2^{5^3}]$ before we change it.    
   \begin{figure}[h]
   \begin{center}
   \includegraphics[scale=0.87]{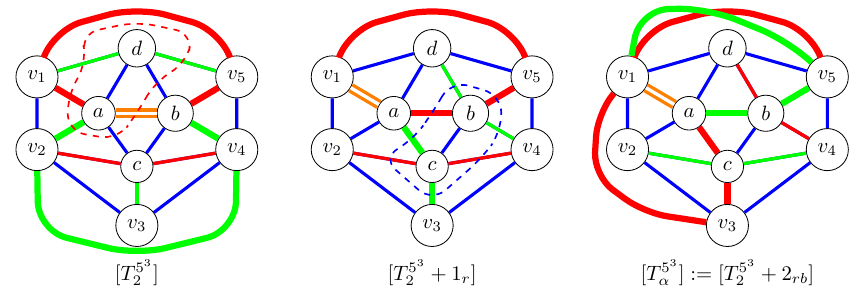}
   \end{center}
\caption{Another way to realize $K_r|_{v_1}^{v^3}$} \label{fig:3deg5a2}  
   \end{figure} 

The third graph above is very interesting and important; so we give it a special name: $[T_\alpha^{5^3}]$.  In this RGB-tiling, all edges along $\Phi$ are blue with three degree 5 vertices inside. What a symmetric structure and edge-coloring! Wait! the graph $[T_\alpha^{5^3}]$ in Figure~\ref{fig:3deg5a2} is not really symmetric. Yes, we do miss a green Kempe chain $K_g|_{v_3}^{v_5}$. Symmetry is not the reason that $K_g|_{v_3}^{v_5}$ exists. The reason can be found in Figure~\ref{fig:3deg5b}. Also notice that to draw $K_r|_{v_1}^{v_3}$ together with $K_r|_{v_1}^{v_5}$ is not different to $K_r|_{v_1}^{v_3}$ together with $K_r|_{v_3}^{v_5}$; or even to draw $K_r|_{v_1}^{v_5}$ together with $K_r|_{v_3}^{v_5}$. {\bf Because what we need is that $v_1$, $v_3$ and $v_5$ are red-connected and also green-connected.}

Again, just for fun, we develop three congruent RGB-tilings on $EP-\{e\}$ for $e=av_1,bv_5,cv_3$ in Figure~\ref{fig:3deg5b}.      
   \begin{figure}[h]
   \begin{center}
   \includegraphics[scale=0.87]{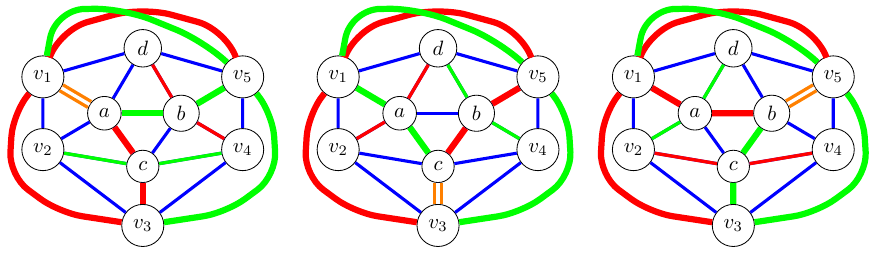}
   \end{center}
   \caption{Three congruent RGB-tilings; All of them are $[T_\alpha^{5^3}]$}     \label{fig:3deg5b} 
   \end{figure}
	\begin{lemma} \label{thm:deg5inT6blue}
Let $EP\in e\mathcal{MPGN}4$ with $a,b,c\in V(EP)$ in a triangle and $\deg(a,b,c)=5$. Three congruent RGB-tilings $[T_\alpha^{5^3}]$ shown in Figure~\ref{fig:3deg5b} must exist.
	\end{lemma}

\subsection{Four degree 5 vertices in a diamond}   \label{sec:4deg5inDiamond}
Finally, we want to finish our interesting question: Can a diamond in $EP$ have all its four vertices degree 5?

   \begin{theorem}    \label{thm:4deg5inDiamond}
Let $a, b, c, d$ be four vertices in a diamond in $EP$. It is impossible that all of them are degree 5.
   \end{theorem}
   \begin{proof}
Now $T\!D:=(\{a,b,c,d\};\deg(a,b,c,d)=5)$ or $5^4$ in short, and $\Omega:=v_1$-$v_2$-$\ldots$-$v_6$-$v_1$ be the 6-cycle of the neighbors of $T\!D$. 

Let us adopt the second graph $[T_2^{5^3}]$ in Figure~\ref{fig:3deg5a} to fit
$\{a,b,c\}$ and $\{a,b,d\}$ and then we obtain the only initial status $[T^{5^4}]$ as in Figure~\ref{fig:4deg5}. Before we proceed the major proof, a very minor check need to be taken care: $v_3\neq v_6$.  Theorem~\ref{RGB2-thm:trivial4cycle} offers a proof. Additionally the red-connectivity of $v_1$, $v_3$ and $v_5$, and then these three vertices are red-disconnected with $v_6$. This fact offers another proof for $v_3\neq v_6$. Wait! We have never checked $v_1\neq v_4$ $v_1\neq v_5$, etc. Actually we should prove these facts before. Lemma~\ref{RGB1-thm:nontrivial3} offers a proof for these simple cases.       
   \begin{figure}[h]
   \begin{center}
   \includegraphics[scale=0.9]{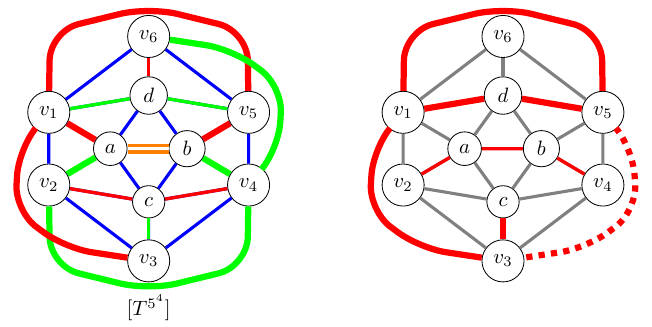}
   \end{center}
   \caption{A diamond with four degree 5's in $EP$}  
   \label{fig:4deg5} 
   \end{figure}

Now we simply re-arrange a new red tiling inside $\Sigma$ shown as the second graph in Figure~\ref{fig:4deg5}. The second graph offers an R-tiling without odd-cycle. If there is a new cycle, then it must cross $\Sigma$. First, $K_r|_{v_1}^{v_5}\cup \{dv_1, dv_5\}$ is an even-cycle. The path $v_2$-$a$-$b$-$v_4$ can not fulfill a bigger cycle because it is blocked by $K_r|_{v_1}^{v_3}$. The last thing to consider is: What about there exists $K_r|_{v_5}^{v_3}$ (red dashed line)? Well, if it exists, then the length is even by the first graph where $v_5$-$v_4$-$v_3$ is length 2 and all blue. Thus the big cycle $K_r|_{v_1}^{v_3}\cup K_r|_{v_5}^{v_3}\cup \{dv_1, dv_5\}$ is even length. By Theorem~\ref{RGB1-thm:4RGBtilingGeneralization}(d), an R-tiling$^\ast$ on an MPG must induce a 4-coloring function. Now we has a contradiction and the proof is done. 
   \end{proof} 

\subsection{Three degree 5 vertices in a triangle, continued}
\label{sec:3degree5cont}

Let us back to $T\!D:=(\{a,b,c\};\deg(a,b,c)=5)$ with $a,b,c$ adjacent. First we need to refer to Lemma~\ref{thm:deg5inT6blue}. Figure~\ref{fig:3deg5b} demonstrates three equivalent RGB-tilings in $\Sigma'$ of $T\!D$, i.e., a necessary skeleton in $\Sigma'$ provided all edges along $\Omega$ blue. Let us redraw that skeleton in $\Sigma'$ but leave every edge inside $\Sigma$ black as the first graph in Figure~\ref{fig:3deg5bb}.  There must be six vertices, say $u_1, u_2, \ldots, u_6$, surrounding $T\!D$. By Theorem~\ref{thm:4deg5inDiamond}, $\deg(d, v_2, v_4)\ge 6$. {\bf So, here we assume the minimum situation $(\ast)$: $\deg(d, v_2, v_4)= 6$ and $\deg(v_1, v_3, v_5)= 5$.} We shall consider a new topic for discussion $\hat{T\!D}$ who has the vertex set $\{a,b,c,d,v_1,\ldots,v_5\}$ and the requirement as the situation $(\ast)$; and then a new $\hat{\Omega}:=u_1$-$u_2$-$\ldots$-$u_6$; also new $\hat{\Sigma}$ and $\hat{\Sigma}'$ inside and outside of $\hat{\Omega}$ respectively. Please, see the second graph in Figure~\ref{fig:3deg5bb}. 

The second graph is the only feasible RGB-tiling on $\Sigma'$ (not only on $\hat{\Sigma}'$) under synonyms. Particularly all edges in $\hat{\Omega}$ must be blue. Now we find a blue canal ring $bCL$ in between $\Omega$ and $\hat{\Omega}$. After performing ECS on this $bCL$, we obtain a new RGB-tiling on $\Sigma'$ shown as the third graph in Figure~\ref{fig:3deg5bb}. Not only that, we can do $\Sigma$-adjustment by coloring paths $v_1$-$a$-$b$-$v_5$ and $v_2$-$c$-$v_4$ red. Finally we obtain a brand new R-tiling without red odd-cycle. Please, check the only red cycle crossing $\hat{\Sigma}$. It must be even length.

	\begin{lemma}  \label{thm:5inTri565656}
Let $EP\in e\mathcal{MPGN}4$ with $a,b,c\in V(EP)$ in a triangle and $\deg(a,b,c)=5$. See Figure~\ref{fig:3deg5bb}. It is impossible that the surrounding vertices along $\Omega:=d$-$v_1$-$v_2$-$v_3$-$v_4$-$v_6$ have degree property: $\deg(d,v_2,v_4)=6$ and $\deg(v_1,v_3,v_5)=5$.	
	\end{lemma}
   \begin{figure}[h]
   \begin{center}
   \includegraphics[scale=0.9]{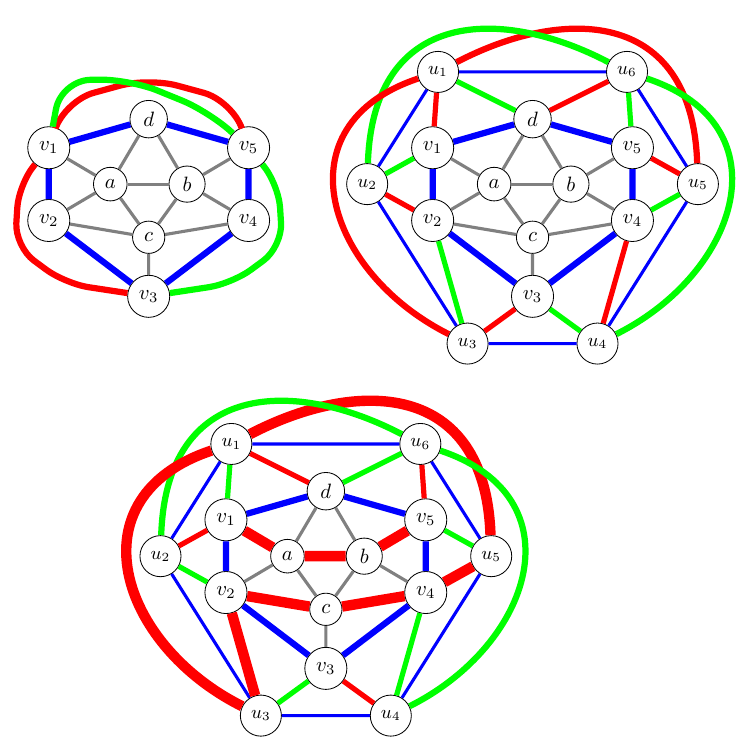}
   \end{center}
   \caption{New $\hat{T\!D}$: a union of $5^3$  and the surrounding $(56)^3$}     \label{fig:3deg5bb} 
   \end{figure}

Our further study shows a more stronger property as follows:     

	\begin{lemma}  \label{thm:5inTri56}
Let $EP\in e\mathcal{MPGN}4$ with $a,b,c\in V(EP)$  in a triangle and $\deg(a,b,c)=5$. Please, refer to the second graph in Figure~\ref{fig:3deg5bb} and the most part of the hypothesis in Lemma~\ref{thm:5inTri565656}.  This time we only assume  $\deg(v_1)=5$ and $deg(d)=6$ in addition, while $deg(v_2, v_4)\ge 6$ (by Theorem~\ref{thm:4deg5inDiamond}) and $deg(v_3, v_5)\ge 5$ are given automatically. It is impossible $EP\in e\mathcal{MPGN}4$.	
	\end{lemma}
	
This new result will be proved in the near future.

\section{No two degree 5 vertices adjacent; We wish.}  \label{sec:deg5Adjacent}

We have a dream to prove the following conjecture that covers all previous results in this paper. Once we thought we did it, but a bug came out. However, we would like demonstrate our false proof.    
   \begin{conjecture}  \label{thm:deg55NoWay}
Are there any two degree 5 vertices adjacent in $EP$? No way!
   \end{conjecture} 

Let $EP\in e\mathcal{MPGN}4$. The given situation is that $T\!D:=(\{a,b\};\deg(a,b)=5)$ or $55$ in short. Then we have  $\Omega:=d$-$v_1$-$v_2$-$c$-$v_4$-$v_5$-$d$.

Now we create a new MPG $\hat{EP}$ from $EP$. We remove vertices $a$ and $b$, and then merge $v_2=v_4$. Notice that $v_2$ and $v_4$ are not adjacent in $\Sigma'$; otherwise the 4-cycle  $a$-$b$-$v_4$-$v_2$-$a$ must form a diamond, but vertex $c$ say no.  Please, see Theorem~\ref{RGB2-thm:trivial4cycle}.  This merging also makes $v_2c=v_4c$ and this fact will cause $v_2c$ and $v_4c$ have same edge-color in the original $EP$. This merging also creates a new 4-outer facet $\Phi:=d$-$v_1$-$v_2$-$v_5$-$d$. In addition, we set a new edge $v_1v_5$ for $\hat{EP}$. 

Thanks for the existence of $[\underline{S_1}]$ on $\Sigma'$ of $EP$. Thus, in Figure~\ref{fig:merging55} we have two synonyms RGB-tilings, (A) and (B), on $\hat{EP}$, which are of course Type A with $e$-diamond and $e=v_1v_5$. 
   \begin{figure}[h] 
   \begin{center}
   \includegraphics[scale=.81]{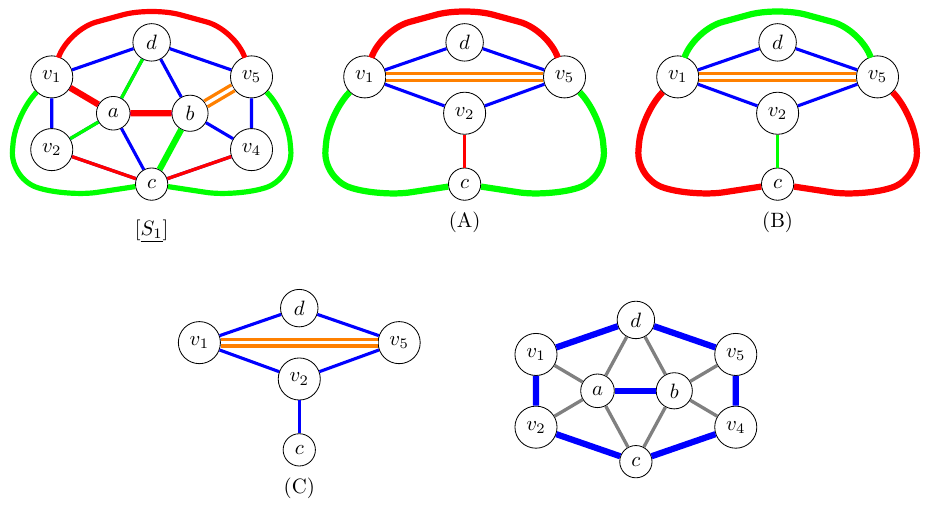}
   \end{center}
   \caption{Merging $v_2=v_4$} \label{fig:merging55}
   \end{figure}
 
The case (C) is the last one we need to consider Type A with $e$-diamond on $\hat{EP}$. However, it does not exist because of the last graph is a 4-colorable B-tiling on $EP$ in Figure~\ref{fig:merging55}. Clearly this B-tiling is restored from (C). 

If we consider Type A condition is a sufficient condition for this $\hat{EP}$ being non-4-colorable, then we obtain a contradiction for $|\hat{EP}| < |EP|$. What a nice proof for  Conjecture~\ref{thm:deg55NoWay}. Unfortunately Type A condition is not a sufficient condition. Please, see False Conjecture~\ref{RGB2-thm:4bsufnes}.

Thank for $ATLAS_N(EP;T\!D)$ in Subsection~\ref{sec:ATLAS}.
We find that $[S_2]$ can offer a Type C $e$-diamond for $\hat{EP}$. Please, see Figure~\ref{fig:merging55b}. By Theorem~\ref{RGB2-thm:4ColorableIfandOnlyIf} or directly by the second graph, $\hat{EP}$ is 4-colorable.
   \begin{figure}[h] 
   \begin{center}
   \includegraphics[scale=.86]{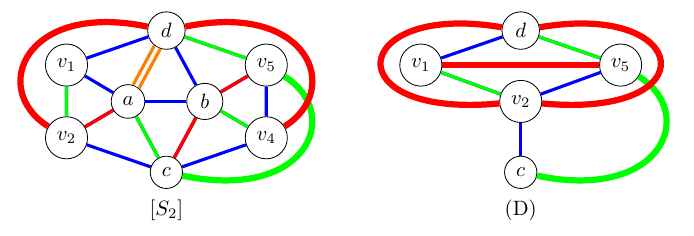}
   \end{center}
   \caption{Merging $v_2=v_4$} \label{fig:merging55b}
   \end{figure}

\section{What are next steps by this renewal approach}

Study $T\!D$, many different $T\!D$. For instance $T\!D$ consists of two or three adjacent vertices of degree 5 or 6. Actually studying the distribution of degrees along $\Omega$ or even the secondary layer $\Omega^2$, especially those vertices of degrees at least 7, is our goal. 

The setting of $\hat{T\!D}$ in Subsection~\ref{sec:3degree5cont} is the minimum situation for $\Omega$. There are many different settings for $\Omega$ to discuss. That will be a new chapter of our study in the near future.

The more vertices in $T\!D$ are or precisely the larger $\sum_{v\in T\!D}\deg(v)$ is, the more complex $ATLAS_N(EP-\{\ast\})$ is. To reduce complexity, Lemma~\ref{thm:2deg5} uses congruence relation between elements in $ATLAS_N(EP-\{\ast\})$.

\bibliographystyle{amsplain}

\end{document}